\documentclass[oneside,english]{amsart}
\usepackage[OT1]{fontenc}
\usepackage[latin9]{inputenc}
\usepackage{mathrsfs}
\usepackage{bm}
\usepackage{amsthm}
\usepackage{amssymb}

\makeatletter
\numberwithin{equation}{section}
\numberwithin{figure}{section}
  \theoremstyle{plain}
  \newtheorem*{thm*}{\protect\theoremname}
\theoremstyle{plain}
\newtheorem{thm}{\protect\theoremname}[section]
  \theoremstyle{definition}
  \newtheorem{example}[thm]{\protect\examplename}
  \theoremstyle{plain}
  \newtheorem{lem}[thm]{\protect\lemmaname}
  \theoremstyle{definition}
  \newtheorem{defn}[thm]{\protect\definitionname}
  \theoremstyle{plain}
  \newtheorem{prop}[thm]{\protect\propositionname}
  \theoremstyle{plain}
  \newtheorem{cor}[thm]{\protect\corollaryname}
  \theoremstyle{remark}
  \newtheorem{rem}[thm]{\protect\remarkname}

\makeatother

\usepackage{babel}
  \providecommand{\corollaryname}{Corollary}
  \providecommand{\definitionname}{Definition}
  \providecommand{\examplename}{Example}
  \providecommand{\lemmaname}{Lemma}
  \providecommand{\propositionname}{Proposition}
  \providecommand{\remarkname}{Remark}
  \providecommand{\theoremname}{Theorem}
\providecommand{\theoremname}{Theorem}

\begin{document}

\title{On the theory of normalized Shintani L-function and its application
to Hecke L-function}
\begin{abstract}
We define the class of normalized Shintani L-functions of several
variables. Unlike Shintani zeta functions, the normalized Shintani
L-function is a holomorphic function. Moreover it satisfies a good
functional equation. We show that any Hecke L-function of a totally
real field can be expressed as a diagonal part of some normalized
Shintani L-function of several variables. This gives a good several
variables generalization of a Hecke L-function of a totally real field.
This also gives a new proof of the functional equation of the Hecke
L-function of a totally real field.
\end{abstract}

\author{minoru hirose}

\thanks{This work was supported by JSPS Grant-in-Aid for Scientific Research
Number 23\textperiodcentered1733}

\maketitle
\tableofcontents{}

\section{Introduction}

The study of the leading term of the Taylor expansion of a Hecke L-function
is one of the important problems in number theory. For a number field
$K$ and a Hecke character $\chi$ of $K$, the Hecke L-function $L(s,\chi)$
is defined by 
\[
L(s,\chi)=\sum_{I}\frac{\chi(I)}{N(I)^{s}}
\]
where $I$ runs through all nonzero ideals of $O_{K}$ and $N(I)$
is the norm of $I$. In this paper, we deal only with Hecke L-functions
of totally real fields.

In this paper, we will consider a generalization to several variables
of the Hecke L-function. Let $K$ be a totally real field of degree
$n$, and let $\chi$ be a Hecke character of $K$. For simplicity,
in this introduction, we assume that the class number of $K$ is equal
to $1$ and that $\chi$ is finite order. Let $(\rho_{\mu})_{\mu=1}^{n}$
be the tuple of all the embeddings of $K$ to $\mathbb{R}$. Since
the norm of a principal ideal $(x)$ is equal to the absolute value
of $\prod_{\mu}\rho_{\mu}(x)$, $L(s,\chi)$ can be written as follows:
\begin{equation}
L(s,\chi)=\sum_{x\in(O_{K}-\{0\})/O_{K}^{\times}}\chi(x)\prod_{\mu=1}^{n}|\rho_{\mu}(x)|^{-s}.\label{eq:onevariableHecke}
\end{equation}
As a generalization of (\ref{eq:onevariableHecke}),  we may consider
a Hecke L-function of several variables as follows:
\[
L((s_{1},\dots,s_{n}),\chi)=\sum_{x\in(O_{K}-\{0\})/O_{K}^{\times}}\chi(x)\prod_{\mu=1}^{n}\left|\rho_{\mu}(x)\right|^{-s_{\mu}}.
\]
(Note that this sum depends on the choice of the representative $x$.)
The main theorem of this paper is an extension of this idea. Let $\infty_{1},\dots,\infty_{n}$
be the infinite places of $K$ corresponding to $\rho_{1},\dots,\rho_{n}$
respectively. For $1\leq\mu\leq n$, we set
\[
\sigma(\mu)=\begin{cases}
0 & \chi\mbox{ is unramified at }\infty_{\mu}\\
1 & \chi\mbox{ is ramified at }\infty_{\mu}.
\end{cases}
\]
For $m\in\mathbb{Z}$, we set $k_{m}=\#\{1\leq\mu\leq n\mid\sigma(\mu)\equiv m\pmod{2}\}$.
Note that $L(s,\chi)(s-m)^{-k_{m}}$ is holomorphic for any non-positive
integer $m$. Let us state the main theorem of this paper.
\begin{thm*}
\label{Theorem_intro_main}There is a natural way to construct holomorphic
function $F:\mathbb{C}^{n}\to\mathbb{C}$ of several variables which
satisfies the following conditions:\end{thm*}
\begin{enumerate}
\item There exist prime ideals $\mathfrak{p},\mathfrak{q}$ of $K$ such
that 
\[
F(s,\dots,s)=L(s,\chi)(1-\chi(\mathfrak{\mathfrak{p}})N(\mathfrak{p})^{1-s})(1-\chi(\mathfrak{q})N(\mathfrak{q})^{-s})\ \ (s\in\mathbb{C}).
\]

\item $F(s_{1},\dots,s_{n})$ has a functional equation.
\item $F(s_{1},\dots,s_{n})$ has zeros at $s_{\mu}-m$ for all $1\leq\mu\leq n$
and non-positive integers $m$ such that $m\equiv\sigma(\mu)\pmod{2}$. 
\end{enumerate}
See Theorem \ref{thm:main_thm} for the precise statement. There are
four purposes of this paper. The first is to establish a theory of
normalized Shintani L-functions. The second is to construct $F(s_{1},\dots,s_{n})$
by normalized Shintani L-functions. The third is to regard the functional
equation of a Hecke L-function as a special case of the functional
equation of the normalized Shintani L-function. The fourth is to give
a theory of fans, which will be used for these purposes.

In this introduction, we give a rough sketch of the construction of
$F(s_{1},\dots,s_{n})$. Firstly we recall the construction of the
Hecke L-function by means of Shintani zeta functions, which was done
by Shintani \cite{MR0427231}. Next, we explain the construction of
the Hecke L-function by means of Shintani L-functions. Finally, we
explain the construction of the Hecke L-function by means of normalized
Shintani L-function. We believe that the Shintani L-function is more
useful than the Shintani zeta function, and that the normalized Shintani
L-function is more useful than the Shintani L-function. 

One of our motivations is the Shintani formula. By using a Hecke L-function
of several variables, we can decompose the derivatives of a Hecke
L-function. If $L((s_{1},\dots,s_{n}))$ is holomorphic at $(s_{1},\dots,s_{n})\in(0,\dots,0)$,
then we have

\begin{equation}
c=\sum_{\mu=1}^{n}c^{(\mu)}\label{eq:shintani-formula2}
\end{equation}

where
\begin{align*}
c & =\frac{\partial}{\partial s}L((s,\dots,s),\chi)\vert_{s=0}\\
c^{(1)} & =\frac{\partial}{\partial s}L((s,0,\dots,0),\chi)\vert_{s=0}\\
 & \vdots\\
c^{(n)} & =\frac{\partial}{\partial s}L((0,\dots,0,s),\chi)\vert_{s=0}.
\end{align*}
Such a decomposition is called a Shintani formula and becomes more
and more important recently. In some cases, the values $c^{(\mu)}$
has more detailed information than $c$ (\cite{MR2499409},\cite{MR2011848}).
They discussed the Shintani formula in \cite{MR2499409} or \cite{MR2011848},
although an L-function of several variables did not appear explicitly.
The equation (\ref{eq:shintani-formula2}) does not necessarily hold
if $L((s_{1},\dots,s_{n}),\chi)$ is not holomorphic at $\bm{s}=(0,\dots,0)$.
Therefore, it is desirable to construct $L((s_{1},\dots,s_{n}))$
by holomorphic functions.

The Shintani zeta function has historically been used to construct
$L((s_{1},\dots,s_{n}),\chi)$. The Shintani zeta function was introduced
by Shintani \cite{MR0427231} for the application to the Hecke L-function.
The Shintani zeta function is defined by the Dirichlet series 
\[
\zeta_{n,r}(\bm{s},A,\bm{x})=\sum_{m_{1},\dots,m_{r}=0}^{\infty}\prod_{\nu=1}^{n}\frac{1}{(a_{\nu1}(x_{1}+m_{1})+\cdots+a_{\nu r}(x_{r}+m_{r}))^{s_{\nu}}}
\]
where $A=(a_{\nu\mu})_{1\leq\nu\leq n,1\leq\mu\leq r}$ with $a_{\nu\mu}>0\ (1\leq\nu,\mu\leq n)$,
$\bm{x}=(x_{1},\dots,x_{r})\in\mathbb{R}_{>0}^{n}$, $\bm{s}=(s_{1},\dots,s_{n})\in\{s\in\mathbb{C}\mid\Re(s)>r/n\}^{n}$.
$\zeta_{n,r}(\bm{s},A,\bm{x})$ is meromorphically continued to the
whole $\bm{s}\in\mathbb{C}^{n}$. We give an example to explain how
to use the Shintani zeta function for the construction of $L((s_{1},\dots,s_{n}),\chi)$.
Let us consider the case where $K=\mathbb{Q}(\sqrt{2})$ and $\chi=1$.
Let us consider the three sets 
\begin{align*}
X_{1} & =\{(m_{1}+1)+(m_{2}+1)(3+2\sqrt{2})\mid m_{1}\geq0,m_{2}\geq0\},\\
X_{2} & =\{(m_{1}+\frac{1}{2})+(m_{2}+\frac{1}{2})(3+2\sqrt{2})\mid m_{1}\geq0,m_{2}\geq0\},\\
X_{3} & =\{(m_{1}+1)\mid m_{1}\geq0\}.
\end{align*}
Since $O_{K}-\{0\}=\bigsqcup_{\epsilon\in O_{K}^{\times}}\epsilon(X_{1}\sqcup X_{2}\sqcup X_{3})$,
we have
\begin{align*}
L((s_{1},s_{2}),\chi)= & \sum_{j=1}^{3}\sum_{x\in X_{j}}\prod_{\mu=1}^{2}\rho_{\mu}(x)^{-s_{\mu}}\\
= & \zeta_{2,2}(\bm{s},\left(\begin{array}{cc}
1 & 3+2\sqrt{2}\\
1 & 3-2\sqrt{2}
\end{array}\right),(1,1))+\zeta_{2,2}(\bm{s},\left(\begin{array}{cc}
1 & 3+2\sqrt{2}\\
1 & 3-2\sqrt{2}
\end{array}\right),(\frac{1}{2},\frac{1}{2}))\\
 & +\zeta_{2,1}(\bm{s},\left(\begin{array}{c}
1\\
1
\end{array}\right),(1)).
\end{align*}
However, the Shintani zeta function is not regular at $\bm{s}=(0,\dots,0)$,
and the equation (\ref{eq:shintani-formula2}) does not hold anymore.
For the first derivatives, one obtains a correct formula by adding
an elementary correction term. However, for the higher derivatives,
it seems difficult to obtain such an elementary expression of the
correction term.

In \cite{SatoHiroseFunctional}, Sato and the author studied the Shintani
L-functions. The Shintani L-function is defined by the Dirichlet series
\begin{equation}
L(\bm{s},A,\bm{x},\bm{y})=\sum_{m_{1},m_{2},\cdots,m_{n}=0}^{\infty}\frac{e^{2\pi i(m_{1}y_{1}+\cdots+m_{n}y_{n})}}{\prod_{1\leq\nu\leq n}(\sum_{1\leq\mu\leq n}a_{\nu\mu}(x_{\mu}+m_{\mu}))^{s_{\nu}}}\label{eq:ShintaniL_x_y}
\end{equation}
where $\bm{x}=(x_{\mu})_{\mu=1}^{n}\in(0,1)^{n}$, $\bm{y}=(y_{\mu})_{\mu=1}^{n}\in(0,1)^{n}$,
$\bm{s}=(s_{\mu})_{\mu=1}^{n}\in\{s\mid\Re(s)>0\}^{n}$, $A=\left(a_{\nu\mu}\right)_{1\leq\nu,\mu\leq n}\in GL_{n}\left(\mathbb{R}\right)$
with $a_{\nu\mu}>0\ (1\leq\nu,\mu\leq n)$. Here $(0,1)$ is the open
interval. What is different from the Shintani zeta functions is the
numerator $e^{2\pi i(m_{1}y_{1}+\cdots+m_{n}y_{n})}$ and the condition
$n=r$, $\bm{x},\bm{y}\in(0,1)^{n}$. Note that excluded (singular)
case when $\bm{y}=\bm{0}$, $L(\bm{s},A,\bm{x},\bm{y})$ reduces to
the Shintani zeta function $\zeta_{n,n}(\bm{s},A,\bm{x})$.

The Shintani L-function is analytically continued to the whole $\mathbb{C}^{n}$
as a holomorphic function. This is one of the most important properties
of the Shintani L-function. The Shintani zeta function does not have
this property. Since the Shintani L-function is holomorphic for all
$\bm{s}\in\mathbb{C}^{n}$, if we can express the Hecke L-function
by the Shintani L-functions, we can get a formula (\ref{eq:shintani-formula2})
without correction terms and also a formula for higher derivatives.

In this paper, we use a slightly different definition of the Shintani
L-function. Let $V$ be an $n$-dimensional vector space over $\mathbb{Q}$.
Let $\mathcal{S}(V)$ denote the set of all functions $\Phi:V\to\mathbb{C}$
which have the following properties:
\begin{itemize}
\item There exists $\mathbb{Z}$-lattice $\mathbb{L}\subset V$ such that
$\Phi(x)=0$ for all $x\notin\mathbb{L}$.
\item There exists $\mathbb{Z}$-lattice $\mathbb{L}\subset V$ such that
$\Phi(x+w)=\Phi(x)$ for all $x\in V$, $w\in\mathbb{L}$. 
\end{itemize}
The symbol $\mathcal{S}(V)$ is used only in this section. Let $u_{1},\dots,u_{n}\in V$
be a basis of $V$. We call the symbol $\Lambda(u_{1},\dots,u_{n})$
a cone. We say that $\Phi\in\mathcal{S}(V)$ is regular with respect
to $\Lambda(u_{1},\dots,u_{n})$ if it satisfies the following conditions:
\begin{enumerate}
\item $\Phi(t_{1}u_{1}+\cdots+t_{n}u_{n})=0$ for all $t_{1},\dots,t_{n}\in\mathbb{Q}$
such that $0\in\{t_{1},\dots,t_{n}\}$.
\item $\lim_{m\to\infty}\frac{1}{m}\sum_{0<x<m}\Phi(v+ux)=0$ for all $v\in V$
and $u\in\{u_{1},\dots,u_{n}\}$.
\end{enumerate}
Very roughly speaking, the condition (1) and (2) corresponds to the
condition $\bm{x}\in(0,1)^{n}$ and $\bm{y}\in(0,1)^{n}$ in (\ref{eq:ShintaniL_x_y})
respectively. Let $\Phi\in\mathcal{S}(V)$ be a function which is
regular with respect to $\Lambda(u_{1},\dots,u_{n})$. Let $\rho=(\rho_{\mu})_{\mu=1}^{n}:V\otimes\mathbb{R}\to\mathbb{R}^{n}$
be an isomorphism. Assume that 
\begin{equation}
\rho(u_{1}),\dots,\rho(u_{n})\in(\mathbb{R}_{>0})^{n}.\label{eq:condition_positive}
\end{equation}
Then $L(\bm{s},\Phi,\Lambda(u_{1},\dots,u_{n}),\rho)$ is defined
by
\[
L(\bm{s},\Phi,\Lambda(u_{1},\dots,u_{n}),\rho)=\sum_{v\in C(u_{1},\dots,u_{n})}\Phi(v)\prod_{\mu=1}^{n}\rho_{\mu}(v)^{-s_{\mu}}
\]
where $C(u_{1},\dots,u_{n})=\{t_{1}u_{1}+\cdots+t_{n}u_{n}\mid t_{1},\dots,t_{n}\in\mathbb{Q}_{>0}\}$
and $\bm{s}=(s_{\mu})_{\mu=1}^{n}\in\{s\mid\Re(s)>0\}^{n}$. The function
$L(\bm{s},\Phi,\Lambda(u_{1},\dots,u_{n}),\rho)$ is analytically
continued to the whole $\bm{s}\in\mathbb{C}^{n}$ as a holomorphic
function. 

In some cases we can express a Hecke L-function by a sum of Shintani
L-functions directly. 
\begin{example}
\label{ex:direct_const}Let $K=\mathbb{Q}(\sqrt{2})$ be a totally
real field and $\chi$ a Hecke character of $K$ defined by $\chi=\chi_{f}\chi_{\infty}$
where 
\[
\chi_{f}((x))=\begin{cases}
1 & x\equiv1\pmod{2}\\
-1 & x\equiv1+\sqrt{2}\pmod{2},
\end{cases}
\]
\[
\chi_{\infty}((x))=\frac{\rho_{1}(x)}{\left|\rho_{1}(x)\right|}\frac{\rho_{2}(x)}{\left|\rho_{2}(x)\right|}.
\]
We define $\Phi\in\mathcal{S}(K)$ by letting 
\[
\Phi(x)=\begin{cases}
\chi_{f}(x) & x\in O_{K}\setminus\left(\sqrt{2}O_{K}\right)\\
0 & \mbox{otherwise}.
\end{cases}
\]
Then $\Phi$ is regular with respect to $\Lambda(2+\sqrt{2},2-\sqrt{2})$.
We define an isomorphism $\rho=(\rho_{1},\rho_{2}):K\otimes\mathbb{R}\to\mathbb{R}^{2}$
by $\rho_{1}(a+b\sqrt{2})=a+b\sqrt{2}$ and $\rho_{2}(a+b\sqrt{2})=a-b\sqrt{2}$
for $a,b\in\mathbb{Q}$. We can express $L(s,\chi)$ by a Shintani
L-function as follows:
\[
L(s,\chi)=L((s,s),\Phi,\Lambda(2+\sqrt{2},2-\sqrt{2}),\rho).
\]

\end{example}
However, in general, we cannot express a Hecke L-function by a sum
of Shintani L-functions directly. For example, let $K$ be a totally
real field and $\chi=1$ be a trivial Hecke character of $K$. Then
we cannot express the Hecke L-function $L(s,\chi)$ by a sum of Shintani
L-functions directly. The condition $\chi=1$ is not crucial. There
are many examples of $K$ and $\chi\neq1$ such that $L(s,\chi)$
cannot be a sum of Shintani L-functions in this way.

In general, to express a Hecke L-function by a sum of Shintani L-functions,
we need a process which we call a regularization. Let us explain this
process by an example. Let $K=\mathbb{Q}(\sqrt{2})$. Let $\chi=1$
be a trivial Hecke character of $K$. In this case, we cannot express
$L(s,\chi)$ by Shintani L-functions directly. For a fractional ideal
$\mathfrak{b}$ of $K$, we define $\Phi_{\mathfrak{b}}\in\mathcal{S}(K)$
by
\[
\Phi_{\mathfrak{b}}(x)=\begin{cases}
1 & x\in\mathfrak{b}^{-1}\\
0 & \mbox{otherwise}.
\end{cases}
\]
Let $\mathfrak{p}=(3+\sqrt{2})$ and $\mathfrak{q}=(5+2\sqrt{2})$
be prime ideals of $K$. We set
\[
\Phi=\Phi_{\mathfrak{p}}-\Phi_{1}-N(\mathfrak{q})(\Phi_{\mathfrak{p}\mathfrak{q}^{-1}}-\Phi_{\mathfrak{q}^{-1}}).
\]
Then $\Phi$ is regular with respect to $\Lambda(1,3+2\sqrt{2})$.
Moreover we have 
\begin{align*}
L(\bm{s},\Phi,\Lambda(1,3+2\sqrt{2})) & =\sum_{x\in K^{\times}/O_{K}^{\times}}\Phi(x)N(x)^{-s}\\
 & =L(s,\chi)(N(\mathfrak{p})^{s}-1)(1-N(\mathfrak{q})^{1-s}).
\end{align*}
In Section \ref{sub:regularization}, we prove some results for this
process.

Sato and the author introduced a normalized Shintani L-function in
\cite{SatoHiroseFunctional}. Let $\sigma:\{1,\dots,n\}\to\{0,1\}$
be any function. In Section \ref{sec:Shintani-L-Lattice}, we will
define $L(\bm{s},\Phi,\Lambda(u_{1},\dots,u_{n}),\rho)$ without the
assumption (\ref{eq:condition_positive}), and we will define normalized
Shintani L-function $L_{\sigma}(\bm{s},\Phi,\Lambda(u_{1},\dots,u_{n}),\rho)$
by
\[
\sum_{g=(g_{\mu})_{\mu=1}^{n}\in\left\{ \pm1\right\} ^{n}}(\prod_{\mu=1}^{n}g_{\mu}^{1-\sigma(\mu)})L(\bm{s},\Phi,\Lambda(u_{1},\dots,u_{n}),\rho_{g})
\]
where $\rho_{g}:V\to\mathbb{R}^{n}$ is defined by $\rho_{g}(x)=(g_{1}\rho(x)_{1},\dots,g_{n}\rho(x)_{n})$.
Note that $L_{\sigma}(s,\Phi,\Lambda(u_{1},\dots,u_{n}),\rho)$ has
no Dirichlet series expression for $n\geq2$. The normalized Shintani
L-function $L_{\sigma}(\bm{s},\Phi,\Lambda(u_{1},\dots,u_{n}),\rho)$
satisfies the following functional equation (\ref{prop:Func_Eq_ShintaniL}):
\[
\Gamma_{\sigma}(\bm{s})L_{\sigma}(\bm{s},\Phi,\Lambda(u_{1},\dots,u_{n}),\rho)=i_{\sigma}\Gamma_{\sigma}(1-\bm{s})L_{\sigma}(\bm{s},\hat{\Phi},\varphi(\Lambda(u_{1},\dots,u_{n})),\rho^{*}).
\]
The Fourier transform $\hat{\Phi}\in\mathcal{S}(V^{*})$, the dual
cone $\varphi(\Lambda(u_{1},\dots,u_{n}))$, and the dual embedding
$\rho^{*}:V^{*}\to\mathbb{R}^{n}$ will be defined later. From the
functional equation, we can see that the normalized Shintani L-function
$L_{\sigma}(\bm{s},\Phi,\Lambda(u_{1},\dots,u_{n}),\rho)$ has trivial
zeros. Let $m_{1},\dots,m_{n}\in\mathbb{Z}_{\leq0}$ be non-positive
integers such that $\sigma(\mu)\equiv m_{\mu}\pmod{2}$ for all $1\leq\mu\leq n$.
Then 
\[
L(\bm{s},\Phi,\Lambda(u_{1},\dots,u_{n}),\rho)\prod_{\mu=1}^{n}(s_{\mu}-m_{\mu})^{-1}
\]
is holomorphic for all $\bm{s}\in\mathbb{C}^{n}$. Therefore it is
desirable to express a Hecke L-function by normalized Shintani L-functions. 

Let us see how to associate a Hecke L-function with the normalized
Shintani L-function. This is not trivial since the normalized Shintani
L-function has no Dirichlet series expression. Let us consider the
same situation as in Example \ref{ex:direct_const}. We set $\epsilon=1+\sqrt{2}$,
$u_{1}=2+\sqrt{2}$, $u_{2}=2-\sqrt{2}$. For $g\in\{\pm1\}^{2}$,
we set $\rho_{(g_{1},g_{2})}(x)=(g_{1}\rho_{1}(x),g_{2}\rho_{2}(x))$
and $K_{g}=\rho_{g}^{-1}(\mathbb{R}_{>0}^{2})$. For $s\in\mathbb{C}$,
we set $\lambda_{s}:K^{\times}\to\mathbb{C}$ by
\[
\lambda_{s}(x)=\Phi(x)N(x)^{-s}
\]
We define cones $(\Lambda_{g})_{g\in\{\pm1\}^{2}}$ by 
\begin{align*}
\Lambda_{(1,1)} & =\Lambda(u_{1},u_{2})\\
\Lambda_{(1,-1)} & =\Lambda(\epsilon u_{2},\epsilon u_{1})\\
\Lambda_{(-1,1)} & =\Lambda(-\epsilon u_{2},-\epsilon u_{1})\\
\Lambda_{(-1,-1)} & =\Lambda(-u_{1},-u_{2}).
\end{align*}
Then we have 
\begin{align*}
4L(s,\chi)= & \sum_{x\in K^{\times}/E_{K}}\chi_{\infty}(x)\lambda_{s}(x)\\
= & \sum_{g\in\{\pm1\}^{2}}g_{1}g_{2}\sum_{x\in K_{g}/E_{K}}\Phi(x)N(x)^{-s}\\
= & \sum_{g\in\{\pm1\}^{2}}g_{1}g_{2}L((s,s),\Phi,\Lambda_{g},\rho_{g}).
\end{align*}
We have 
\begin{align*}
 & L((s,s),\Phi,\Lambda_{(1,-1)},\rho_{(1,-1)})+L((s,s),\Phi,\Lambda_{(1,1)},\rho_{(1,-1)})\\
= & L((s,s),\Phi,\Lambda(\epsilon u_{2},u_{2}),\rho_{(1,-1)})-L((s,s),\Phi,\Lambda(\epsilon u_{1},u_{1}),\rho_{(1,-1)})\\
= & 0.
\end{align*}
The last equation follows from the invariance of $\lambda_{s}(x)$
under the action of $E_{K}$. Note that, $L((s_{1},s_{2}),\Phi,\Lambda_{(1,-1)},\rho_{(1,-1)})$
has a Dirichlet series expression, but, $L((s_{1},s_{2}),\Phi,\Lambda_{(1,1)},\rho_{(1,-1)})$
has no such a Dirichlet series expression. Similarly we have
\begin{equation}
L((s,\dots,s),\Phi,\Lambda_{g},\rho_{g})=g_{1}g_{2}L((s,\dots,s),\Phi,\Lambda_{(1,1)},\rho_{g})\label{eq:another_cone}
\end{equation}
for all $g\in\{\pm1\}^{2}$. Therefore we have
\begin{align*}
4L(s,\chi) & =\sum_{g\in\{\pm1\}^{2}}g_{1}g_{2}L((s,s),\Phi,\Lambda_{(1,1)},\rho_{g})\\
 & =\sum_{g\in\{\pm1\}^{2}}L((s,s),\Phi,\Lambda_{(1,1)},\rho_{g})\\
 & =L_{\sigma}((s,s),\Phi,\Lambda_{(1,1)},\rho)
\end{align*}
where $\sigma(1)=\sigma(2)=1$. If we set $F(s_{1},s_{2})=\frac{1}{4}L((s_{1},s_{2}),\Phi,\Lambda_{(1,1)},\rho)$,
then $F(s_{1},s_{2})$ has the following properties:
\begin{enumerate}
\item $F(s,s)=L(s,\chi)$ for all $s\in\mathbb{C}$.
\item $F(s_{1},s_{2})$ has a functional equation.
\item $F(s_{1},s_{2})$ has zeros along $s_{\mu}-m=0$ for all negative
odd integers $m$ and $1\leq\mu\leq2$.
\end{enumerate}
This is a special case of Theorem \ref{thm:main_thm}. Note that $L((s_{1},s_{2}),\Phi,\Lambda_{(1,1)},\rho)$
satisfies the condition (1) but does not satisfy the condition (2). 

The situation is more complicated for general $K$. To deal with those
cases, we introduce the notion of fans. Let $C_{m}(K)$ be the free
$\mathbb{Z}$-module generated by the symbols
\[
\Lambda(a_{1},\dots,a_{m})\;\;\;\;\;(a_{1},\dots,a_{m}\in K^{\times}).
\]
We call $\Lambda(a_{1},\dots,a_{m})$ a cone, and an element of $C_{m}(K)$
a fan. We define the boundary homomorphism $\partial_{m}:C_{m}(K)\to C_{m-1}(K)$
by
\[
\partial_{m}\Lambda(a_{1},\dots,a_{m})=\sum_{j=1}^{m}(-1)^{j+1}\Lambda(a_{1},\dots,\hat{a}_{j},\dots,a_{m}).
\]
We define $I_{n}(K^{\times},E_{K})$ to be the submodule generated
by 
\[
\{\Lambda(\epsilon a_{1},\dots,\epsilon a_{n})-\Lambda(a_{1},\dots,a_{n})\mid\epsilon\in E_{k},\ a_{1},\dots,a_{n}\in K^{\times}\}.
\]
The key to the proof of \ref{eq:another_cone} was the fact that $\Lambda_{g}-g_{1}g_{2}\Lambda_{(1,1)}\in I_{2}(K^{\times},E_{K})+\ker\partial_{2}$.
We will generalize this fact to the higher dimensional cases in Proposition
\ref{prop:another_cubic_fandamental_domain}. We use Proposition \ref{prop:another_cubic_fandamental_domain}
for the expression of the Hecke L-function by normalized Shintani
L-functions. 

Since the normalized Shintani L-function has a functional equation,
it is expected that the functional equation of the Hecke L-function
follows from that of the normalized Shintani L-function. We give a
new proof of the functional equation of the Hecke L-function by this
method in Section \ref{sub:Functional-equation-of-Hecke}.

For this new proof, we need some results on the fans. We identify
$K$ and $K^{*}$ by an inner product $\left\langle x,y\right\rangle =\mathrm{Tr}_{K/\mathbb{Q}}(xy)$
for $x,y\in K$, where ${\rm Tr}_{K/\mathbb{Q}}$ is the field trace.
For an $n$-dimensional cone $\Lambda(a_{1},\dots,a_{n})$ with $a_{1},\dots,a_{n}\in K$,
we define the dual cone $\varphi(\Lambda(a_{1},\dots,a_{n}))$ by
$\Lambda(b_{1},\dots,b_{n})$ where $b_{1},\dots,b_{n}\in K$ is the
dual basis of $a_{1},\dots,a_{n}$. We define the dual fan $\varphi(\sum_{\mu}c_{\mu}\Lambda_{\mu})$
by $\sum_{\mu}c_{\mu}\varphi(\Lambda_{\mu}).$ The dual fans are important
since they appear in the functional equations of the Shintani L-functions.
We define $W\subset C_{n}(V)$ to be the submodule generated by
\[
\{\Lambda(a_{1},\dots,a_{n})\mid a_{1},\dots,a_{n}\in K^{\times},\ a_{1},\dots,a_{n}\mbox{ are linearly dependent}\}.
\]
In Section \ref{sec:Fan}, we define the notion of a fundamental domain
fan. Roughly speaking, a fundamental domain fan is an element of $C_{n}(K)$
which is equal to a certain fundamental domain for $(\mathbb{R}_{>0})^{n}/E_{K}$
modulo $I_{n}(K^{\times},E_{K})+W+\ker\partial_{n}$. We use a fundamental
domain fan for the expression of the Hecke L-function by a Shintani
L-function. Therefore we have to show that the dual fan of a fundamental
domain fan is also a fundamental domain fan. This is one of the main
theorem of this paper (Theorem \ref{thm:dual_of_fandamental_domain_is_fandamental_domain}).
To prove this result, we treat fans abstractly and algebraically in
Section \ref{sec:Fan}. Let us see an example of this result. Let
$K=\mathbb{Q}(\sqrt{2})$ and $\epsilon=3-2\sqrt{2}$. We have $E_{K}=\{\epsilon^{m}\mid m\in\mathbb{Z}\}$.
Then $\Lambda(1,3-2\sqrt{2})$ is a fundamental domain for $(\mathbb{R}_{>0})^{n}/E_{K}$.
Since
\[
{\vphantom{\left(\begin{array}{c}
1\\
1
\end{array}\right)}}^{t}\left(\begin{array}{cc}
1 & 3-2\sqrt{2}\\
1 & 3+2\sqrt{2}
\end{array}\right)^{-1}=\left(\begin{array}{cc}
\frac{4+3\sqrt{2}}{8} & \frac{-\sqrt{2}}{8}\\
\frac{4-3\sqrt{2}}{8} & \frac{\sqrt{2}}{8}
\end{array}\right),
\]
the dual fan of this fan is as follows
\begin{align*}
\varphi(\Lambda(1,3-2\sqrt{2})) & =\Lambda(x,-x\epsilon),\ \ \ \ \ x=\frac{(4+3\sqrt{2})}{8}.
\end{align*}
Of course, $\Lambda(x,-x\epsilon)$ is not a fundamental domain for
$(\mathbb{R}_{>0})^{n}/E_{K}$. However $\Lambda(x,-x\epsilon)$ is
a fundamental domain fan since $\Lambda(x,-x\epsilon)=\Lambda(1,\epsilon)+$$w_{1}+w_{2}+w_{3}$
where
\begin{align*}
w_{1} & =\Lambda(\epsilon,x\epsilon)-\Lambda(1,x)\in I_{n}(K^{\times},E_{K}),\\
w_{2} & =\Lambda(x\epsilon,-x\epsilon)\in W,\\
w_{3} & =\Lambda(x,-x\epsilon)-\Lambda(1,\epsilon)-\Lambda(\epsilon,x\epsilon)+\Lambda(1,x)-\Lambda(x\epsilon,-x\epsilon)\in\ker\partial_{2}.
\end{align*}
In higher dimensional cases, the situation is more complicated.

\section*{Basic notations and definitions}

For an $n$-dimensional vector space $V$ over $\mathbb{Q}$ or $\mathbb{R}$,
an orientation of $V$ is a nonzero function $r:V^{n}\to\{0,\pm1\}$
that satisfies
\[
r(f(u_{1}),\dots,f(u_{n}))=r(u_{1},\dots,u_{n})\times\begin{cases}
1 & \det(f)>0\\
0 & \det(f)=0\\
-1 & \det(f)<0
\end{cases}
\]
for all $u_{1},\dots,u_{n}\in V$ and all linear maps $f:V\to V$.
For an orientation $r_{1}$ of $V_{1}$ and an orientation $r_{2}$
of $V_{2}$, $r_{1}\times r_{2}$ is an orientation of $V_{1}\times V_{2}$
which satisfies 
\[
(r_{1}\times r_{2})(u_{1},\dots,u_{n},v_{1},\dots,v_{n})=r_{1}(u_{1},\dots,u_{n})\times r_{2}(v_{1},\dots,v_{n})
\]
where $u_{1},\dots,u_{n}\in V_{1}$ and $v_{1},\dots,v_{n}\in V_{2}$.
We use the following notation.

\begin{align*}
e(z) & =e^{2\pi iz},\\
\Gamma_{\mathbb{R}}(s) & =\pi^{-s/2}\Gamma(\frac{s}{2}),\\
\Gamma_{\mathbb{C}}\left(s\right) & =2\left(2\pi\right)^{-s}\Gamma(s),\\
\mathbb{R}_{>0} & =\{t\in\mathbb{R}\mid t>0\}.
\end{align*}
For $\sigma:\{1,\dots,n\}\to\{0,1\}$, we define $\Gamma_{\sigma}$
and $i_{\sigma}$ by
\begin{align*}
\Gamma_{\sigma}((s_{1},\dots,s_{n})) & =\prod_{\mu=1}^{n}\Gamma_{\mathbb{R}}(\sigma(\mu)+s_{\mu})\\
i_{\sigma} & =\prod_{\mu=1}^{n}i^{\sigma(\mu)}.
\end{align*}
For $g=(g_{\mu})_{\mu=1}^{n}$, we set $\mathbb{R}_{g}=\prod_{\mu=1}^{n}(g_{\mu}\mathbb{R}_{>0})\subset\mathbb{R}^{n}$. 

For a totally real field $K$, $N=N_{K/\mathbb{Q}}$ is the field
norm, ${\rm Tr}={\rm Tr}_{K/\mathbb{Q}}$ is the field trace, and
$E_{K}$ is the multiplicative group of totally positive units of
$K$. The characteristic function of a set $U$ is denoted by $\bm{1}_{U}$.

\section*{Basic notations for fan}

Let $V$ be an $n$-dimensional vector space over $\mathbb{Q}$. We
set $C_{n}(V)$ the $\mathbb{Z}$-module generated by the symbols
\[
\Lambda(a_{1},\dots,a_{n})\;\;\;\;\;(a_{1},\dots,a_{n}\in V\setminus\{0\}).
\]
We call $\Lambda(a_{1},\dots,a_{n})$ a cone, and an element of $C_{n}(V)$
a fan. A cone $\Lambda(u_{1},\dots,u_{n})$ is called simple if $u_{1},\dots,u_{n}$
are linearly independent. Let us fix an orientation $r:V^{n}\to\{0,\pm1\}$.
For a cone $\Lambda=\Lambda(u_{1},\dots,u_{n})\in C_{n}(V)$, we define
the characteristic function $\mathfrak{C}(\Lambda):V\to\mathbb{Z}$
by
\[
\mathfrak{C}(\Lambda)=\begin{cases}
r(u_{1},\dots,u_{n})\times\bm{1}_{C(u_{1},\dots,u_{n})} & \Lambda\mbox{ is simple}\\
0 & \mbox{otherwise}
\end{cases}
\]
where $C(u_{1},\dots,u_{n})=\{t_{1}u_{1}+\cdots+t_{n}u_{n}\mid(t_{\mu})_{\mu=1}^{n}\in\mathbb{Q}_{>0}^{n}\}\subset V$.
For a cone $\Lambda=\Lambda(u_{1},\dots,u_{n})\in C_{n}(V)$, we define
the dual cone $\varphi(\Lambda)\in C_{n}(V^{*})$ by
\[
\varphi(\Lambda)=\begin{cases}
\Lambda(v_{1},\dots,v_{n}) & \Lambda\mbox{ is simple}\\
0 & \mbox{otherwise}
\end{cases}
\]
where $v_{1},\dots,v_{n}$ is the dual basis of $u_{1},\dots,u_{n}$.
For a fan $\mathbb{B}=\sum_{j}c_{j}\Lambda_{j}\in C_{n}(V)$, we define
$\mathfrak{C}(\mathbb{B}):V\to\mathbb{Z}$ by $\mathfrak{C}(\mathbb{B})=\sum_{j}c_{j}\mathfrak{C}(\Lambda_{j})$,
and define the dual fan $\varphi(\mathbb{B})\in C_{n}(V^{*})$ by
$\varphi(\mathbb{B})=\sum_{j}c_{j}\varphi(\Lambda_{j})$.

\section{\label{sec:Shintani-L-Lattice}The theory of Shintani L-functions}

Let $V$ be a vector space over $\mathbb{Q}$. We denote by $\mathbb{A}_{f}$
the finite adele ring of $\mathbb{Q}$. We set $\hat{\mathbb{Z}}=\prod_{p}\mathbb{Z}_{p}\subset\mathbb{A}_{f}$.
For $R=\mathbb{A}_{f}$ or $R=\mathbb{R}$, we set $V(R)=V\otimes_{\mathbb{Z}}R$.
We denote by $\mathcal{S}(V(R))$ the set of Schwartz-Bruhat functions
on $V(R)$. In this section, we define a function $L(\bm{s},\Phi,\mathbb{B},\rho)$
for $\bm{s}\in\mathbb{C}^{n}$, $\Phi\in\mathcal{S}(V(\mathbb{A}_{f}))$,
a fan $\mathbb{B}\in C_{n}(V)$, and an isomorphism $\rho:V(\mathbb{R})\to\mathbb{R}^{n}$,
under some conditions.

For $f\in\mathcal{S}(V(R))$ where $R$ is $\mathbb{R}$ or $\mathbb{A}_{f}$,
we define the Fourier transform $\hat{f}\in\mathcal{S}(V^{*}(R))$
by 
\[
\hat{f}(t)=\int_{V(R)}f(v)\psi(vt)dv
\]
where $\psi:R\to\mathbb{C}^{\times}$ is the standard additive character,
i.e. $\psi$ on $\mathbb{R}$ is defined by
\[
\psi(x)=e(-x),
\]
and $\psi$ on $\mathbb{A}_{f}$ is defined by 
\[
\psi(a)=e(m)
\]
where $m\in\mathbb{Q}$ such that $a-m\in\hat{\mathbb{Z}}$.

\subsection{\label{sub:Shintani-L-function-of-lattice}The function $L_{\sigma}(\bm{s},f,\rho)$}

Let $\mathrm{MAP}(V,\mathbb{C})$ be the set of all functions from
$V$ to $\mathbb{C}$. For $w\in V,$ we define an operator $\Delta_{w}:\mathrm{MAP}(V,\mathbb{C})\to\mathrm{MAP}(V,\mathbb{C})$
by
\[
(\Delta_{w}f)(v)=f(v-w).
\]
Let $X(V)$ be a $\mathbb{C}$-algebra of operators on $V$ generated
by $\{\Delta_{w}\mid w\in V\}$. We set

\[
\mathscr{U}'(V)=\{f:V\to\mathbb{C}\mid f(v)=0\mbox{ for all but finitely many }v\in V\}
\]
and

\[
\mathscr{U}(V)=\{f:V\to\mathbb{C}\mid\mbox{there exists }\Delta\in X(V)\setminus\{0\}\mbox{ such that }\Delta(f)\in\mathscr{U}'(V)\}.
\]
Note that $V$ can be naturally embedded into $V(\mathbb{A}_{f})$. 

Let $\mathcal{M}(V^{*}(\mathbb{R}))$ be the field of fractions of
the ring of real analytic functions on $V^{*}(\mathbb{R})$. We define
$F:\mathscr{U}(V)\to\mathcal{M}(V^{*}(\mathbb{R}))$ as follows. This
construction is a modified version of the Solomon-Hu pairing \cite{MR1794257}.
We define a homomorphism of $\mathbb{C}$-algebra $P:X(V)\to\mathcal{M}(V^{*}(\mathbb{R}))$
by
\[
P(\Delta_{w})(t)=e(i\left\langle t,w\right\rangle ).
\]
It is obvious that $P$ is injective. For $f\in\mathscr{U}'(V)$,
we define $F(f)\in\mathcal{M}(V^{*}(\mathbb{R}))$ by
\[
F(f)(t)=\sum_{v\in V}f(v)e(i\left\langle t,v\right\rangle )\ \ (t\in V^{*}(\mathbb{R})).
\]
We extend the definition of $F(f)\in\mathcal{M}(V^{*}(\mathbb{R}))$
to $f\in\mathscr{U}(V)$ by
\[
F(f)=\frac{F'(\Delta(f))}{P(\Delta)}
\]
where $\Delta\in X(V)$ is any nonzero element such that $\Delta(f)\in\mathscr{U}'(V)$.
This definition is well-defined since
\[
F(\Delta(f))=P(\Delta)F(f)
\]
for all $f\in\mathscr{U}'(V)$ and $\Delta\in X(V)$.

The following two lemmas are obvious from the definition.
\begin{lem}
\label{lem:F-tensor-f}Let $V=V_{1}\times\cdots\times V_{m}$ and
$f_{j}\in\mathscr{U}(V_{j})$ for $j=1,\dots,m$. Then we have
\[
F(f_{1}\otimes\cdots\otimes f_{m})=F(f_{1})\otimes\cdots\otimes F(f_{m}).
\]

\end{lem}

\begin{lem}
\label{lem:F-series}Let $f\in\mathscr{U}(V)$ and $t\in V^{*}(\mathbb{R})$.
Assume that 
\[
\sum_{v\in V}f(v)e(i\left\langle t,v\right\rangle )
\]
converges absolutely. Then we have
\[
\sum_{v\in V}f(v)e(i\left\langle t,v\right\rangle )=F(f)(t).
\]

\end{lem}
Let us fix an isomorphism $\rho=(\rho_{1},\dots,\rho_{n}):V(\mathbb{R})\to\mathbb{R}^{n}$.
We denote by $\rho^{*}:V^{*}(\mathbb{R})\to\mathbb{R}^{n}$ the isomorphism
dual to $\rho$. We define the Shintani L-function of $f\in\mathscr{U}(V)$
under the condition that $F(f)\in\mathcal{S}(V^{*}(\mathbb{R})).$
For $\bm{t}=(t_{\mu})\in(\mathbb{R}^{\times})^{n}$, $\bm{s}=(s_{\mu})\in\mathbb{C}^{n}$,
and a map $\sigma:\{1,\dots,n\}\to\{0,1\}$, we set
\[
\left|\bm{t}\right|_{\sigma}^{\bm{s}}=\prod_{1\leq\mu\leq n}\left(\frac{t_{\mu}}{|t_{\mu}|}\right)^{\sigma(\mu)}\left|t_{\mu}\right|^{s_{\mu}}.
\]

Now let us define the functions $L(g,\bm{s},f,\rho)$ and $L_{\sigma}(\bm{s},f,\rho)$.
\begin{defn}
Let $f\in\mathscr{U}(V)$ be a map such that $F(f)\in\mathcal{S}(V^{*}(\mathbb{R}))$.
Let $g=(g_{\mu})_{\mu=1}^{n}\in\{\pm1\}^{n}$. For $\bm{s}=(s_{\mu})_{\mu=1}^{n}\in\mathbb{C}^{n}$
such that $\Re(s_{\mu})>0$ for all $1\leq\mu\leq n$, we define $L(g,\bm{s},f,\rho)$
by
\[
L(g,\bm{s},f,\rho)=\frac{2^{n}}{\Gamma_{\mathbb{C}}(s_{1})\cdots\Gamma_{\mathbb{C}}(s_{n})}\int_{\mathbb{R}_{g}}F(f)((\rho^{*})^{-1}(t_{1},\dots,t_{n}))\prod_{\mu=1}^{n}(g_{\mu}t_{\mu})^{s_{\mu}}\frac{dt_{1}}{t_{1}}\cdots\frac{dt_{n}}{t_{n}}
\]
where
\[
\mathbb{R}_{g}=\prod_{\mu=1}^{n}(g_{\mu}\mathbb{R}_{>0})\subset\mathbb{R}^{n}.
\]
In the case where $g=(1)_{\mu=1}^{n}$, we simply denote $L(g,\bm{s},f,\rho)$
as $L(\bm{s},f,\rho)$.
\end{defn}
Note that we have
\[
L(g,\bm{s},f,\rho)=L(\bm{s},f,\rho_{g})\prod_{\mu=1}^{n}g_{\mu}
\]
where $\rho_{g}(x)=\rho(x)\cdot c_{g}$ with $c_{g}=(\delta_{\nu\mu}g_{\mu})_{\nu,\mu=1}^{n}\in M_{n}(\mathbb{R})$. 

The function $L(\bm{s},f,\rho)$ (and so $L(g,\bm{s},f,\rho)$) is
holomorphically continued to whole $\bm{s}\in\mathbb{C}^{n}$. This
fact follows from the following integral expression
\begin{align*}
L(\bm{s},f,\rho)= & \prod_{\mu=1}^{n}\frac{2}{\Gamma_{\mathbb{C}}(s_{\mu})(e(s_{\mu})-1)}\int_{L^{n}}F(f)((\rho^{*})^{-1}(t_{1},\dots,t_{n}))t_{1}^{s_{1}}\cdots t_{n}^{s_{n}}\frac{dt_{1}}{t_{1}}\cdots\frac{dt_{n}}{t_{n}}
\end{align*}
where $L$ is a contour which starts and ends at $+\infty$ and circles
the origin once counterclockwise with sufficiently small radius.

\begin{defn}
Let $f\in\mathscr{U}(V)$ be a map such that $F(f)\in\mathcal{S}(V^{*}(\mathbb{R}))$.
Let $\sigma:\{1,\dots,n\}\to\{0,1\}$ be a map. For $\bm{s}=(s_{\mu})\in\mathbb{C}^{n}$
such that $\Re(s_{\mu})>0$ for all $1\leq\mu\leq n$, we define $L_{\sigma}(\bm{s},f,\rho)$
and $\hat{L}_{\sigma}(\bm{s},f,\rho)$ by
\begin{align*}
L_{\sigma}(\bm{s},f,\rho) & =\frac{2^{n}}{\Gamma_{\mathbb{C}}(s_{1})\cdots\Gamma_{\mathbb{C}}(s_{n})}\int_{(\mathbb{R}^{\times})^{n}}F(f)((\rho^{*})^{-1}(t_{1},\dots,t_{n}))\left|\bm{t}\right|_{\sigma}^{\bm{s}}\frac{dt_{1}}{t_{1}}\cdots\frac{dt_{n}}{t_{n}}\\
 & =\sum_{g\in\{\pm1\}^{n}}L(g,\bm{s},f,\rho)\prod_{\mu}g_{\mu}^{\sigma(\mu)}
\end{align*}
and 
\begin{align*}
\hat{L}_{\sigma}(\bm{s},f,\rho) & =\Gamma_{\sigma}(\bm{s})L_{\sigma}(\bm{s},f,\rho).
\end{align*}

\end{defn}
From the equation 
\[
\frac{2}{\Gamma_{\mathbb{C}}(s)}\int_{0}^{\infty}e(ivt)t^{s}\frac{dt}{t}=v^{-s}\ \ \ \ (v>0,\ \Re(s)>0),
\]
we have the following lemma.
\begin{lem}
\label{lem:series-expression}Assume that $f(v)=0$ for all $v\notin\rho^{-1}(\mathbb{R}_{g})$.
Then we have
\end{lem}
\[
L(g,\bm{s},f,\rho)=(\prod_{\mu=1}^{n}g_{\mu})\sum_{v\in V}f(v)\prod_{\mu=1}^{n}(g_{\mu}\rho_{\mu}(v))^{-s_{\mu}}.
\]

Let us state the functional equation of $L_{\rho}(\bm{s},f,\rho)$.
We use Tate's local functional equation for $\mathbb{R}$ in \cite{MR546607}.
\begin{lem}
\label{lem:Tate}Let $\phi\in\mathcal{S}(\mathbb{R}^{n})$ be a Schwartz-Bruhat
function and $\hat{\phi}\in\mathcal{S}(\mathbb{R}^{n})$ the Fourier
transform of $\phi$. Then we have
\[
\frac{\int_{(\mathbb{R}^{\times})^{n}}\phi(t)\left|t\right|_{\sigma}^{s}d^{\times}t}{\Gamma_{\sigma}(s)}=i_{\sigma}\frac{\int_{(\mathbb{R}^{\times})^{n}}\hat{\phi}(t)\left|t\right|_{\sigma}^{1-s}d^{\times}t}{\Gamma_{\sigma}(1-s)}
\]
where 
\[
d^{\times}t=\prod_{\mu=1}^{n}\frac{dt_{\mu}}{\left|t_{\mu}\right|}
\]
 is the Haar measure of $(\mathbb{R}^{\times})^{n}$.
\end{lem}
Let $f\in\mathscr{U}(V)$. Assume that there exists $\bar{f}\in\mathscr{U}(V^{*})$
such that $F(\bar{f})$ equals to the Fourier transform of $F(f)$.
Here, we consider the Haar measure on $V(\mathbb{R})$ as the pull-back
of the Lebesgue measure on $\mathbb{R}^{n}$ by $\rho$. By applying
the Lemma \ref{lem:Tate} to $\phi=F(f)\circ(\rho^{*})^{-1}$, we
get
\[
\hat{L}_{\sigma}(s,f,\rho)=i_{1-\sigma}\hat{L}_{\sigma}(1-\bm{s},\bar{f},\rho^{*}).
\]

\subsection{Regularity of $\Phi\in\mathcal{S}(V(\mathbb{A}_{f}))$ with respect
to $\mathbb{B}\in C_{n}(V)$}

In the previous section, we defined $L_{\sigma}(\bm{s},f,\rho)$ for
$f\in\mathscr{U}(V)$ such that $F(f)\in\mathcal{S}(V^{*}(\mathbb{R}))$.
In this section, we show that 
\[
\Phi\mathfrak{C}(\mathbb{B})\in\mathscr{U}(V)
\]
for $\Phi\in\mathcal{S}(V(\mathbb{A}_{f}))$ and $\mathbb{B}\in C_{n}(V)$,
where $\Phi\mathfrak{C}(V)$ is defined by
\[
(\Phi\mathfrak{C}(\mathbb{B}))(v)=\Phi(v)\cdot\mathfrak{C}(\mathbb{B})(v)\ \ \ \ (v\in V).
\]
Thus we can define $F(\Phi\mathfrak{C}(\mathbb{B}))\in\mathcal{M}(V^{*}(\mathbb{R}))$
for all $\Phi\in\mathcal{S}(V(\mathbb{A}_{f}))$ and $\mathbb{B}\in C_{n}(V)$,
but, $F(\Phi\mathfrak{C}(\mathbb{B}))\in\mathcal{S}(V^{*}(\mathbb{R}))$
does not always hold. In this section, we introduce the notion of
the regularity of $\Phi$ with respect to $\mathbb{B}$, which is
a sufficient condition for $F(\Phi\mathfrak{C}(\mathbb{B}))\in\mathcal{S}(V^{*}(\mathbb{R}))$. 
\begin{defn}
For $\Phi\in\mathcal{S}(V(\mathbb{A}_{f}))$ and a simple cone $\Lambda=\Lambda(u_{1},\dots,u_{n})\in C_{n}(V)$,
we say that $(\Phi,\Lambda)$ satisfies Condition $P_{1}$ if 
\[
\Phi(x_{1}u_{1}+\cdots+x_{n}u_{n})=0
\]
for all $(x_{k})_{k=1}^{n}\in\mathbb{A}_{f}^{n}$ such that $0\in\{x_{1},\dots,x_{n}\}$.
\end{defn}

\begin{defn}
For $\Phi\in\mathcal{S}(V(\mathbb{A}_{f}))$ and a simple cone $\Lambda=\Lambda(u_{1},\dots,u_{n})\in C_{n}(V)$,
we say that $(\Phi,\Lambda)$ satisfies Condition $P_{2}$ if 
\[
\int_{\mathbb{A}_{f}}\Phi(v+xu_{j})dx=0
\]
for all $1\leq j\leq n$ and $v\in V(\mathbb{A}_{f})$.\end{defn}
\begin{lem}
\label{Rem:p1p2}For $\Phi\in\mathcal{S}(V(\mathbb{A}_{f}))$ and
a simple cone $\Lambda\in C_{n}(V)$, $(\Phi,\Lambda)$ satisfies
Condition $P_{2}$ if and only if $(\hat{\Phi},\varphi(\Lambda))$
satisfies Condition $P_{1}$.\end{lem}
\begin{proof}
Set $\Lambda=\Lambda(u_{1},\dots,u_{n})$ and $\varphi(\Lambda)=\Lambda(t_{1},\dots,t_{n})$.
If $(\Phi,\Lambda)$ satisfies Condition $P_{2}$ then $(\hat{\Phi},\varphi(\Lambda))$
satisfies Condition $P_{1}$ since
\begin{align*}
\hat{\Phi}(y_{2}t_{2}+\cdots+y_{n}t_{n}) & =\int_{\mathbb{A}_{f}^{n}}\Phi(x_{1}u_{1}+\cdots+x_{n}u_{n})\psi(x_{2}y_{2}+\cdots+x_{n}y_{n})dx_{1}\cdots dx_{n}\\
 & =\int_{\mathbb{A}_{f}^{n-1}}(\int_{\mathbb{A}_{f}}\Phi(x_{1}u_{1}+\cdots+x_{n}u_{n})dx_{1})\psi(x_{2}y_{2}+\cdots+x_{n}y_{n})dx_{2}\cdots dx_{n}\\
 & =0
\end{align*}
for all $y_{2},\dots,y_{n}\in\mathbb{A}_{f}$. Assume that $(\hat{\Phi},\varphi(\Lambda))$
satisfies Condition $P_{1}$. It is enough to prove that $\phi=0$
where $\phi:\mathbb{A}_{f}^{n-1}\to\mathbb{C}$ is defined by 
\[
\phi(x_{2},\dots,x_{n})=\int_{\mathbb{A}_{f}}\Phi(x_{1}u_{1}+x_{2}u_{2}+\cdots+x_{n}u_{n})dx_{1}.
\]
The Fourier transform $\hat{\phi}:\mathbb{A}_{f}^{n-1}\to\mathbb{C}$
of $\phi$ are given by
\begin{align*}
\hat{\phi}(y_{2},\dots,y_{n}) & =\int_{\mathbb{A}_{f}^{n-1}}\phi(x_{2},\dots,x_{n})\psi(x_{2}y_{2}+\cdots+x_{n}y_{n})dx_{2}\cdots dx_{n}\\
 & =\int_{\mathbb{A}_{f}^{n}}\Phi(x_{1}u_{1}+x_{2}u_{2}+\cdots+x_{n}u_{n})\psi(x_{2}y_{2}+\cdots+x_{n}y_{n})dx_{1}\cdots dx_{n}\\
 & =\hat{\Phi}(y_{2}t_{2}+\cdots+y_{n}t_{n}).
\end{align*}
Since $(\hat{\Phi},\varphi(\Lambda))$ satisfies Condition $P_{1}$,
we have $\hat{\Phi}(y_{2}t_{2}+\cdots+y_{n}t_{n})=0$ and $\hat{\phi}=0.$
Thus we have $\phi=0$.
\end{proof}

\begin{defn}
\label{Def_regular_cone}For $\Phi\in\mathcal{S}(V(\mathbb{A}_{f}))$
and a simple cone $\Lambda=\Lambda(u_{1},\dots,u_{n})$ where $u_{1},\dots,u_{n}\in V$,
we say that $\Phi$ is regular with respect to $\Lambda$ if $(\Phi,\Lambda)$
satisfies Condition $P_{1}$ and Condition$P_{2}$.
\end{defn}
For $\mathbb{B}\in C_{n}(V)$, there exist a finite set $A$ of distinct
cones and a function $c:A\to\mathbb{Z}\setminus\{0\}$ such that $\mathbb{B}=\sum_{\Lambda\in A}c(\Lambda)\Lambda$.
We define the set $X(\mathbb{B})\subset C_{n}(V)$ of simple cones
by
\[
X(\mathbb{B})=\{\Lambda\mid\Lambda\in A\mbox{ and }\Lambda\mbox{ is a simple cone}\}.
\]

\begin{defn}
\label{Def_regular_fan}For $\Phi\in\mathcal{S}(V(\mathbb{A}_{f}))$
and $\mathbb{B}\in C_{n}(V)$, $\Phi$ is said to be regular with
respect to $\mathbb{B}$ if $\Phi$ is regular with respect to $\Lambda$
for all $\Lambda\in X(\mathbb{B})$.
\end{defn}
For $u_{1},\dots,u_{n}\in V$ (resp. $\mathbb{B}\in C_{n}(V)$), we
denote by $\mathcal{R}(V;u_{1},\dots,u_{n})$ (resp. $\mathcal{R}(V,\mathbb{B})$)
the set of all $\Phi\in\mathcal{S}(V(\mathbb{A}_{f}))$ which is regular
with respect to $\Lambda(u_{1},\dots,u_{n})$ (resp. $\mathbb{B}\in C_{n}(V)$).
We will prove later that $F(\Phi,\mathbb{B},\sigma)\in\mathcal{S}(V(\mathbb{A}_{f}))$
for all $\Phi\in\mathcal{R}(V,\mathbb{B})$. 
\begin{lem}
\label{Lem:regulartiy_from_tensor}Let $(u_{1},\dots,u_{n})$ be a
basis of $V$. For $1\leq j\leq n$, let $V_{j}$ be a $1$-dimensional
vector space generated by $u_{j}$. Then we have
\[
\bigotimes_{j=1}^{n}\mathcal{R}(V_{j};u_{j})=\mathcal{R}(V;u_{1},\dots u_{n})
\]
as a subspace of $\mathcal{S}(V(\mathbb{A}_{f}))$.
\end{lem}
This lemma is obvious from the definition.
\begin{lem}
For $\mathbb{B}\in C_{n}(V)$ and $\Phi\in\mathcal{S}(V(\mathbb{A}_{f}))$,
we have $\Phi\mathfrak{C}(\mathbb{B})\in\mathscr{U}(V)$. In other
words, there exists $\Delta\in X(V)\setminus\{0\}$ such that $\Delta(\Phi\mathfrak{C}(\mathbb{B}))(v)=0$
for all but finitely many $v\in V$.\end{lem}
\begin{proof}
It is enough to consider the case where $\mathbb{B}=\Lambda(u_{1},\dots,u_{n})$
is a simple cone. Assume that an orientation of $V$ are given by
$r(u_{1},\dots,u_{n})=1$. There exists positive integer $m>0$ such
that $\Phi(v+mu_{j})=\Phi(v)$ for all $v\in V$ and $1\leq j\leq n$.
Set
\[
\Delta=\prod_{j=1}^{n}(1-\Delta_{mu_{j}}).
\]
Then, for $v=k_{1}mu_{1}+\cdots+k_{n}mu_{n}\in V$ with $k_{1},\dots,k_{r}\in\mathbb{Q}$,
we have
\[
\Delta(\Phi\mathfrak{C}(\mathbb{B}))(v)=\begin{cases}
\Phi(v) & 0<k_{j}\leq1\mbox{ for all }1\leq j\leq n\\
0 & \mbox{otherwise}.
\end{cases}
\]
It follows that $\Delta(\Phi\mathfrak{C}(\mathbb{B}))(v)=0$ for all
but finitely many $v\in V$, and the lemma is proved.
\end{proof}
For $\mathbb{B}\in C_{n}(V)$ and $\Phi\in\mathcal{S}(V(\mathbb{A}_{f}))$,
we set
\[
F(\Phi,\mathbb{B})=F(\Phi\mathfrak{C}(\mathbb{B}))\in\mathcal{M}(V^{*}(\mathbb{R})).
\]
We need some lemmas to prove that $F(\Phi,\mathbb{B})\in S(V^{*}(\mathbb{R}))$
for all $\Phi\in\mathcal{R}(V,\mathbb{B})$.
\begin{lem}
\label{lem:F_tensor}Let $(v_{1},\dots,v_{n})$ be a basis of $V$,
and let $(\hat{v}_{1},\dots,\hat{v}_{n})$ be the dual basis of $(v_{1},\dots,v_{n})$.
For $1\leq j\leq n$, let $V_{j}$ be a vector space generated by
$v_{j}$, and let $V_{j}^{*}$ be a vector space generated by $\hat{v}_{j}$.
Assume that an orientation $r$ of $V$ are given by $r=r_{1}\times\cdots\times r_{n}$
where $r_{1},\dots,r_{n}$ are orientations of $V_{1},\dots,V_{n}$.
Then for $\Phi_{j}\in\mathcal{S}(V_{j}(\mathbb{A}_{f}))$, we have
\[
F(\Phi,\Lambda(v_{1},\dots,v_{n}))=\bigotimes_{j=1}^{n}F(\Phi_{j},\Lambda(v_{j}))\in\bigotimes_{j=1}^{n}\mathcal{S}(V_{j}^{*}(\mathbb{R}))
\]
where
\[
\Phi=\bigotimes_{j=1}^{n}\Phi_{j}\in\mathcal{S}(V(\mathbb{A}_{f})).
\]
\end{lem}
\begin{proof}
By Lemma \ref{lem:F-tensor-f}, it is enough to prove that
\[
(\bigotimes_{j=1}^{n}\Phi_{j})\mathfrak{C}(\Lambda(v_{1},\dots,v_{n}))=\bigotimes_{j=1}^{n}\Phi_{j}\mathfrak{C}(\Lambda(v_{j})),
\]
and this is obvious from the definition of $\mathfrak{C}$. \end{proof}
\begin{lem}
\label{lem:one_dim_case}Let $V$ be a $1$-dimensional vector space.
For all $u\in V\setminus\{0\}$ and $\Phi\in\mathcal{R}(V;u)$, we
have
\[
F(\Phi,\Lambda(u))\in\mathcal{S}(V^{*}(\mathbb{R})).
\]
\end{lem}
\begin{proof}
Take $w\in V\setminus\{0\}$ and a positive integer $N$ which satisfy
the following conditions.
\begin{enumerate}
\item $r(u)=r(w)$
\item $\Phi(kw)=0$ for all $k\in\mathbb{Q}\setminus\mathbb{Z}$
\item $\Phi((k+N)w)=\Phi(kw)$ for all $k\in\mathbb{Z}$.
\end{enumerate}
From the condition $\Phi\in\mathcal{R}(V;u)$, we have $\Phi(Nw)=0$
and $\sum_{k=1}^{N}\Phi(kw)=0$. For all $k\in\mathbb{Q}$, we have
\[
(1-\Delta_{Nw})(\Phi\mathfrak{C}(\Lambda(u)))(kw)=\begin{cases}
\Phi(kw) & 1\leq k\leq N-1\\
0 & \mbox{otherwise}.
\end{cases}
\]
Hence we have
\begin{align}
F(\Phi,\Lambda(u))(t) & =\frac{1}{1-e(iN\left\langle w,t\right\rangle )}\sum_{k=1}^{N-1}\Phi(kw)e(ik\left\langle w,t\right\rangle )\nonumber \\
 & =\frac{1}{1-e^{-Nz}}\sum_{k=1}^{N-1}\Phi(kw)e^{-kz}\label{eq:temp_F}
\end{align}
where we set $z=2\pi\left\langle w,t\right\rangle $. Therefore $F(\Phi,\Lambda(u))(t)$
is rapidly decreasing for $z\to\pm\infty$. Moreover $F(\Phi,\Lambda(u))(t)$
is real analytic at $z=0$ since $\sum_{k=1}^{N-1}\Phi(kw)=0$. Hence
$F(\Phi,\Lambda(u))\in\mathcal{S}(V^{*}(\mathbb{R}))$.
\end{proof}

\begin{lem}
\label{lem:reg_u}Let $(u_{1},\dots,u_{n})$ be a basis of $V$. For
all $\Phi\in\mathcal{R}(V;u_{1},\dots,u_{n})$, we have $F(\Phi,\Lambda(u,\dots,u_{n}))\in\mathcal{S}(V^{*}(\mathbb{R}))$.\end{lem}
\begin{proof}
By Lemma \ref{Lem:regulartiy_from_tensor}, it is enough to prove
that 
\[
F(\bigotimes_{j=1}^{n}\Phi_{j},\Lambda(u_{1},\dots,u_{n}))\in\mathcal{S}(V^{*}(\mathbb{R}))
\]
for all $(\Phi_{j})_{j=1}^{n}\in\prod_{j=1}^{n}\mathcal{R}(V_{j};u_{j})$.
By Lemma \ref{lem:F_tensor}, it is enough to prove that
\[
F(\Phi_{j},\Lambda(u_{j}))\in S(V_{j}^{*}(\mathbb{R}))
\]
for all $1\leq j\leq n$. Thus the lemma follows from Lemma \ref{lem:one_dim_case}.
\end{proof}
The following proposition is obvious from Lemma \ref{lem:reg_u}.
\begin{prop}
For all $\mathbb{B}\in C_{n}(V)$ and $\Phi\in\mathcal{R}(V,\mathbb{B})$,
we have
\[
F(\Phi,\mathbb{B})\in\mathcal{S}(V^{*}(\mathbb{R})).
\]

\end{prop}
Now we can define $L_{\sigma}(\bm{s},\Phi,\mathbb{B})$. For a map
$\sigma:\{1,\dots,r\}\to\{0,1\}$, $g\in\{\pm1\}^{n}$,$\bm{s}\in\mathbb{C}^{n}$,
$\mathbb{B}\in C_{n}(V)$, and $\Phi\in\mathcal{R}(V,\mathbb{B})$,
we set
\[
L(g,\bm{s},\Phi,\mathbb{B},\rho)=L(g,\bm{s},\Phi\mathfrak{C}(\mathbb{B}),\rho)
\]
and
\[
L_{\sigma}(\bm{s},\Phi,\mathbb{B},\rho)=L_{\sigma}(\bm{s},\Phi\mathfrak{C}(\mathbb{B}),\rho).
\]

\begin{lem}
\label{lem:sign_change_Lfunc}Let $(u_{1},\dots,u_{n})$ be a basis
of $V$. Then for $e_{1},\dots,e_{n}\in\{\pm1\}$ and $\Phi\in\mathcal{R}(V;u_{1},\dots,u_{n})$,
we have
\[
F(\Phi,\Lambda(u_{1},\dots,u_{n}))=F(\Phi,\Lambda(e_{1}u_{1},\dots,e_{n}u_{n})).
\]
\end{lem}
\begin{proof}
It is enough to prove in the case where $e_{1}=-1$ and $e_{j}=1$
for all $2\leq j\leq n$. We set $f=\Phi\mathfrak{C}(\Lambda(u_{1},\dots,u_{n})-\Lambda(-u_{1},\dots,u_{n}))$.
Let $k$ be a nonzero positive integer such that $\Phi(v+ku_{1})=\Phi(v)$
for all $v\in V$. Then we have
\[
(\Delta_{ku_{1}}-1)f=0.
\]
Since $P(\Delta_{ku_{1}}-1)\neq0$, we have $F(f)=0$, and the lemma
is proved.
\end{proof}

\subsection{\label{sub:Functional-equation-of-lattice}The functional equation
of $L_{\sigma}(\bm{s},\Phi,\mathbb{B},\rho)$}

In this section we show the functional equation of $L_{\sigma}(s,\Phi,\mathbb{B},\rho)$.
For this purpose, we need the Fourier transform of $F(\Phi,\mathbb{B})$.
We consider compatible measures $dv_{\infty}$ on $V(\mathbb{R})$
and $dv_{f}$ on $V(\mathbb{A}_{f})$, i.e., 
\[
\int_{V(\mathbb{A})/V}dv_{\infty}dv_{f}=1.
\]

\begin{prop}
We have
\[
\hat{F}(\Phi,\mathbb{B})=i^{-n}F(\hat{\Phi},\varphi(\mathbb{B})).
\]
\end{prop}
\begin{proof}
By Lemma \ref{lem:F_tensor}, it is enough to consider the case where
$n=1$ and $\mathbb{B}=\Lambda(u)$ with $u\in V$. Since the statements
of the proposition does not depend on the choice of the orientation
$r$, we assume that $r(u)=1$. Take $w\in V\setminus\{0\}$ and a
positive integer $N$ which satisfy the following conditions:
\begin{enumerate}
\item $r(u)=r(w)$,
\item $\Phi(kw)=0$ for all $k\in\mathbb{Q}\setminus\mathbb{Z}$,
\item $\Phi((k+N)w)=\Phi(kw)$ for all $k\in\mathbb{Z}$.
\end{enumerate}
Let $\hat{w}$ be a unique element of $V^{*}$ determined by $\left\langle w,\hat{w}\right\rangle =1$.
Then $\hat{\Phi}$ satisfies the following conditions:
\begin{enumerate}
\item $\hat{\Phi}(\frac{m}{N}\hat{w})=0$ for all $m\in\mathbb{Q}\setminus\mathbb{Z}$,
\item $\hat{\Phi}(\frac{m+N}{N}\hat{w})=\hat{\Phi}(\frac{m}{N}\hat{w})$
for all $m\in\mathbb{Z}$.
\end{enumerate}
Therefore we have 
\[
F(\hat{\Phi},\varphi(\mathbb{B}))(k'w)=\frac{1}{1-e(ik')}\sum_{m=1}^{N-1}\hat{\Phi}(\frac{m}{N}\hat{w})e(i\frac{k'm}{N})
\]
for $k'\in\mathbb{R}$. Since the statement of the proposition does
not depend on the choice of compatible measures, we may assume that
\[
\mu(V(\mathbb{R})/\mathbb{L})=\mu(\mathbb{L}\otimes_{\mathbb{Z}}\prod_{p}\mathbb{Z}_{p})=1
\]
where $\mathbb{L}=w\mathbb{Z}\subset V$. Let us compute $\hat{F}(\Phi,\mathbb{B})$.
We have
\[
F(\Phi,\mathbb{B})(m'\hat{w})=\frac{1}{1-e(iNm')}\sum_{k=1}^{N-1}\Phi(kw)e(ikm')
\]
for $m'\in\mathbb{R}$, and we have 
\[
\hat{F}(\Phi,\mathbb{B})(k'w)=\int_{-\infty}^{\infty}\phi(m')dm'
\]
for $k'\in\mathbb{R}$, where 
\[
\phi(m')=F(\Phi,\mathbb{B})(m'\hat{w})e(-k'm').
\]
We compute this integral by the residue theorem. Note that $\phi(t)$
has poles at 
\[
t\in(\frac{i}{N}\mathbb{Z})\setminus(i\mathbb{Z}).
\]
For $m\in\mathbb{Z}$, let $h(m)$ be the residue of $\phi(t)$ at
$t=im/N$. Then $h(m)$ is given by
\begin{align*}
h(m) & =\frac{1}{2\pi N}\sum_{k=1}^{N-1}\Phi(kw)e(-\frac{km}{N}-i\frac{k'm}{N})\\
 & =\frac{1}{2\pi}\hat{\Phi}(-\frac{m}{N}\hat{w})e(-i\frac{k'm}{N}).
\end{align*}
Since $\phi(t-i)=e(ik')\phi(t)$ for all $t\in\mathbb{C}$, if $k'\neq0$
then we have
\begin{align*}
\hat{F}(\Phi,\mathbb{B})(k'w) & =\frac{1}{1-e(ik')}\left(\int_{-\infty}^{\infty}\phi(t)dt-\int_{-\infty-i}^{\infty-i}\phi(t)dt\right)\\
 & =\frac{-2\pi i}{1-e(ik')}\sum_{m=1}^{N-1}h(-m)\\
 & =\frac{-i}{1-e(ik')}\sum_{m=1}^{N-1}\hat{\Phi}(\frac{m}{N}\hat{w})e(i\frac{k'm}{N})\\
 & =-iF(\Phi,\varphi(\mathbb{B}))(k'w).
\end{align*}
In the case $v=0$, the proposition follows from the continuity of
$\hat{F}(\Phi,\mathbb{B})$.
\end{proof}
By this lemma we have the following proposition.
\begin{prop}
\label{prop:Func_Eq_ShintaniL}We have
\[
\Gamma_{\sigma}(\bm{s})L_{\sigma}(\bm{s},\Phi,\mathbb{B},\rho)=i_{\sigma}^{-1}\Gamma_{\sigma}(\bm{1}-\bm{s})L_{\sigma}(\bm{1}-\bm{s},\hat{\Phi},\varphi(\mathbb{B}),\rho^{*})
\]
and
\[
\hat{L}_{\sigma}(\bm{s},\Phi,\mathbb{B},\rho)=i_{\sigma}^{-1}\hat{L}(\bm{1}-\bm{s},\hat{\Phi},\varphi(\mathbb{B}),\rho^{*}).
\]
Here, we consider the Haar measure on $V(\mathbb{A}_{f})$ as compatible
with the pull-back of the Lebesgue measure on $\mathbb{R}^{n}$ by
$\rho$. \end{prop}
\begin{example}
For $A=\left(a_{\nu\mu}\right)_{1\leq\nu,\mu\leq n}\in GL_{n}\left(\mathbb{R}\right)$
and $\bm{x},\bm{y}\in\{t\in\mathbb{Q}\mid0<t<1\}^{n}$, we define
$\Phi_{\bm{x},\bm{y}}\in\mathcal{S}(\mathbb{A}_{f}^{n})$ by
\[
\Phi_{\bm{x},\bm{y}}(v)=\begin{cases}
\psi(\bm{y}\cdot\bm{m}) & \mbox{there exists }\bm{m}\in\hat{\mathbb{Z}}^{n}\mbox{ such that }v=\bm{x}+\bm{m}\\
0 & \mbox{otherwise}
\end{cases}
\]
and $\rho_{A}:\mathbb{Q}^{n}\to\mathbb{R}^{n}$ by 
\[
\rho_{A}(v)=Av.
\]
We define a cone $\Lambda\in C_{n}(\mathbb{Q}^{n})$ by $\Lambda=\Lambda(u_{1},\dots,u_{n})$
where $u_{1}=(1,0,\dots,0),\dots,u_{n}=(0,\dots,0,1)$. Then $L_{\sigma}(\bm{s},A,\bm{x},\bm{y})=2^{-n}L_{\sigma}(\bm{s},\Phi,\Lambda,\rho_{A})$
where the left side is the Shintani L-function defined in \cite{SatoHiroseFunctional}.
Since the inverse Fourier transform of $\Phi_{\bm{x},\bm{y}}$ are
given by $e(-\bm{x}\cdot\bm{y})\Phi_{\bm{y},\bm{1}-\bm{x}}$, we have
\[
\hat{L}_{\sigma}(\bm{s},A,\bm{x},\bm{y})=i_{\sigma}e(-\bm{x}\cdot\bm{y})\hat{L}_{\sigma}(\bm{1}-\bm{s},(A^{-1})^{t},\bm{y},\bm{1}-\bm{x}).
\]

\end{example}

\section{\label{sec:Fan}The theory of fans}

In this section, we give a theory of fans. In Section \ref{sub:Definition-of-fan}--\ref{sub:Dual-fan-of-fundamental-domain-fan},
we deal with the abstract theory of fans. In Section \ref{sub:Indicator-function-of-fan},
we prove some property of the characteristic functions of fans. The
hardest part in this section is the proof of Theorem \ref{thm:dual_of_fandamental_domain_is_fandamental_domain}.

For $1\leq i\leq n$, the coface map $d^{i}:\{1,\dots,n-1\}\rightarrow\{1,\dots,n\}$
is defined by
\[
d^{i}(j)=\begin{cases}
j & 1\leq j\leq i-1\\
j+1 & i\leq j\leq n.
\end{cases}
\]
For $0\leq n$, $\mathfrak{S}_{n}$ is the symmetric group of $\{1,\dots,n\}$.
For $m,i\in\{1,\dots,n\}$ and $\sigma\in\mathfrak{S}_{n-1}$, we
define $u(m,i,\sigma)\in\mathfrak{S}_{n}$ by
\[
u(m,i,\sigma)(k)=\begin{cases}
i & k=m\\
d^{i}(\sigma(j)) & 1\leq k<m\\
d^{i}(\sigma(j-1)) & m<k\leq n.
\end{cases}
\]
Note that
\[
\mathfrak{S}_{n}=\{u(m,i,\sigma)\mid1\leq i\leq n,\ \sigma\in\mathfrak{S}_{n-1}\}
\]
for all $1\leq m\leq n$. We have
\[
\mathrm{sgn}(u(m,i,\sigma))=(-1)^{m+i}\mathrm{sgn}(\sigma)
\]
for all $m,i\in\{1,\dots,n\}$ and $\sigma\in\mathfrak{S}_{n-1}$. 

For a partially ordered set (poset, for short) $Y$, we use the following
notation. 
\begin{align*}
S_{k}(Y) & =\{f:\{1,\dots,k\}\to Y\mid f(i)<f(i+1)\mbox{ for all }1\leq i\leq k-1\}\ \ \ (k\geq0),\\
P(Y) & =\{U\subset Y\mid U\neq\emptyset,\ U\mbox{ is totally orderd}\}.
\end{align*}
We regard $P(Y)$ as a poset ordered by inclusion. Sometimes we write
$f$ as $\left[f(1),\dots,f(k)\right]$ for $f\in S_{k}(Y)$ especially
in examples.

\subsection{\label{sub:Definition-of-fan}Definition of fans}

Let $X$ be a non-empty set. We define $C_{m}(X)$ as a $\mathbb{Z}$-module
generated by the symbols
\[
\Lambda(a_{1},\dots,a_{m})\;\;\;\;\;(a_{1},\dots,a_{m}\in X).
\]
We call the symbol $\Lambda(a_{1},\dots,a_{m})$ a \emph{cone} and
call an element of $C_{m}(X)$ a \emph{fan}. We understand that $C_{0}(X)\simeq\mathbb{Z}$.
Sometimes we use the notation $\Lambda(a_{j})_{j=1}^{m}$ for $\Lambda(a_{1},\dots,a_{m})$.
For non-negative integers $m$ and $n$, we identify $C_{m}(X)\otimes C_{n}(X)$
and $C_{m+n}(X)$ by
\[
\Lambda(a_{1},\dots,a_{m})\otimes\Lambda(b_{1},\dots,b_{n})=\Lambda(a_{1},\dots,a_{m},b_{1},\dots,b_{n}).
\]
We define the boundary homomorphism $\partial_{m}:C_{m}(X)\rightarrow C_{m-1}(X)$
by
\[
\partial_{m}\Lambda(a_{1},\dots,a_{m})=\sum_{i=1}^{m}(-1)^{i+1}\Lambda(a_{1},\dots,\hat{a_{i}},\dots,a_{m}).
\]
Then we get a chain complex
\begin{equation}
\cdots\rightarrow C_{2}(X)\rightarrow C_{1}(X)\rightarrow\mathbb{Z}\rightarrow0.\label{eq:resolution}
\end{equation}
For $x\in C_{n}(X)$, if $\partial_{n}x=0$ then $\partial_{n+1}(\Lambda(a)\otimes x)=x-\Lambda(a)\otimes(\partial_{n}x)=x$
for any $a\in X$. Therefore the sequence (\ref{eq:resolution}) is
exact. 

Let us introduce the notion of a conical system. Roughly speaking,
conical system is a object which can be treated as a set of cones.
\begin{defn}
Let $Y$ be a poset. A map $\mathcal{F}:\bigcup_{k=0}^{\infty}S_{k}(Y)\to\bigcup_{k=0}^{\infty}C_{k}(X)$
is said to be a \emph{conical system} on $Y$ if it satisfies the
following conditions.
\begin{enumerate}
\item $\mathcal{F}(S_{k}(Y))\subset C_{k}(X)$ for all $k\geq0$. 
\item $\partial_{k}\mathcal{F}(f)=\sum_{j=1}^{k}(-1)^{j+1}F(f\circ d^{j})$
for all $f\in S_{k}(Y)$.
\item $\mathcal{F}(\emptyset)=1$ where $\emptyset$ is the unique element
of $S_{0}(Y)$.
\end{enumerate}
\end{defn}
Let $Y$ be a poset and $h:Y\to X$ a map. We define 
\[
\mathcal{S}_{h}:\bigcup_{k=0}^{\infty}S_{k}(Y)\to\bigcup_{k=0}^{\infty}C_{k}(X)
\]
 by
\[
\mathcal{S}_{h}(f)=\Lambda(h(f(1)),\dots,h(f(m)))
\]
for $f\in S_{m}(Y)$. Then $\mathcal{S}_{h}$ is a conical system
on $Y$. We say that $\mathcal{S}_{h}$ is the \emph{standard conical
system} attached to $h$. Let us see another example of a conical
system.
\begin{example}
Let $Y=\{1,2,3\}$ be an ordered set. Let $x_{1},x_{2},x_{3},x_{12},x_{13},x_{23},x_{123}\in X$.
We define a map $\mathcal{F}:\bigcup_{k=0}^{\infty}S_{k}(Y)\to\bigcup_{k=0}^{\infty}C_{k}(X)$
by
\begin{align*}
\mathcal{F}(\emptyset) & =1\\
\mathcal{F}(\left[1\right]) & =\Lambda(x_{1})\\
\mathcal{F}(\left[2\right]) & =\Lambda(x_{2})\\
\mathcal{F}(\left[3\right]) & =\Lambda(x_{3})\\
\mathcal{F}(\left[1,2\right]) & =\Lambda(x_{1},x_{12})-\Lambda(x_{2},x_{12})\\
\mathcal{F}(\left[1,3\right]) & =\Lambda(x_{1},x_{13})-\Lambda(x_{3},x_{13})\\
\mathcal{F}(\left[2,3\right]) & =\Lambda(x_{2},x_{23})-\Lambda(x_{3},x_{23})\\
\mathcal{F}(\left[1,2,3\right]) & =\Lambda(x_{1},x_{12},x_{123})+\Lambda(x_{2},x_{23},x_{123})+\Lambda(x_{3},x_{13},x_{123})\\
 & \ \ \ \ -\Lambda(x_{1},x_{13},x_{123})-\Lambda(x_{2},x_{12},x_{123})-\Lambda(x_{3},x_{23},x_{123}).
\end{align*}
Then $\mathcal{F}$ is a conical system on $Y$. This is a special
case of Corollary \ref{cor:T1_conic_sytem}.
\end{example}
The next lemma will be used in the proof of Lemma \ref{lem:dual_equation}.
The proof is obvious from a simple calculation.
\begin{lem}
\label{lem:fusion_product}Let $\mathcal{F}$ and $\mathcal{G}$ be
conical systems on $\{1,\dots,n\}$. We define $x\in C_{n}(X)$ by
\[
x=\sum_{k=0}^{n}\sum_{\substack{f\in S_{k}(\{1,\dots,n\})\\
g\in S_{n-k}(\{1,\dots,n\})\\
\mathrm{Im}f\cap\mathrm{Im}g=\emptyset
}
}(-1)^{k}\mathrm{sgn}(f,g)\mathcal{F}(f)\otimes\mathcal{G}(g)
\]
where $\mathrm{sgn}(f,g)$ is defined by
\[
\mathrm{sgn}(f,g)=(-1)^{t}
\]
with
\[
t=\#\{(i,j)\in\{1,\dots,k\}\times\{1,\dots,n-k\}\mid f(i)>g(j)\}.
\]
Then we have $\partial_{n}x=0$. \end{lem}
\begin{example}
In the case $n=3$, $x\in\ker\partial_{3}$ can be written as follows:
\begin{align*}
x & =\mathcal{F}(\emptyset)\otimes\mathcal{G}(\left[1,2,3\right])\\
 & \ \ -\mathcal{F}(\left[1\right])\otimes\mathcal{G}(\left[2,3\right])+\mathcal{F}(\left[2\right])\otimes\mathcal{G}(\left[1,3\right])-\mathcal{F}(\left[3\right])\otimes\mathcal{G}(\left[1,2\right])\\
 & \ \ +\mathcal{F}(\left[1,2\right])\otimes\mathcal{G}(\left[3\right])-\mathcal{F}(\left[1,3\right])\otimes\mathcal{G}(\left[2\right])+\mathcal{F}(\left[2,3\right])\otimes\mathcal{G}(\left[1\right])\\
 & \ \ -\mathcal{F}(\left[1,2,3\right])\otimes\mathcal{G}(\emptyset).
\end{align*}
\end{example}
\begin{defn}
Let $Y$ be a poset and $\mathcal{F}$ a conical system on $P(Y)$.
For $m\geq0$ and $h\in S_{m}(A)$, we define $T_{1}(\mathcal{F},h)\in C_{m}(X)$
by
\[
T_{1}(\mathcal{F},h)=\sum_{\sigma\in\mathfrak{S}_{m}}{\rm sgn}(\sigma)\mathcal{F}(g_{\sigma,h})
\]
where $g_{\sigma,h}\in S_{m}(P(Y))$ is defined by
\[
g_{\sigma,h}(j)=\{h(\sigma(k))\mid1\leq k\leq j\}.
\]
\end{defn}
\begin{example}
Let us consider the case where $m=3$. Put $h_{i}=h(i)\in Y$ for
$i=1,2,3$. Then $T_{1}(\mathcal{F},h)\in C_{3}(X)$ is given as follows:
\begin{align*}
T_{1}(\mathcal{F},h) & =\mathcal{F}(\left[\{h_{1}\},\{h_{1},h_{2}\},\{h_{1},h_{2},h_{3}\}\right])+\mathcal{F}(\left[\{h_{2}\},\{h_{2},h_{3}\},\{h_{1},h_{2},h_{3}\}\right])\\
 & \ \ \ +\mathcal{F}(\left[\{h_{3}\},\{h_{1},h_{3}\},\{h_{1},h_{2},h_{3}\}\right])-\mathcal{F}(\left[\{h_{1}\},\{h_{1},h_{3}\},\{h_{1},h_{2},h_{3}\}\right])\\
 & \ \ \ -\mathcal{F}(\left[\{h_{2}\},\{h_{1},h_{2}\},\{h_{1},h_{2},h_{3}\}\right])-\mathcal{F}(\left[\{h_{3}\},\{h_{2},h_{3}\},\{h_{1},h_{2},h_{3}\}\right]).
\end{align*}

\end{example}
The boundary of $T_{1}(\mathcal{F},h)$ is given in the following
lemma. 
\begin{lem}
Let $Y$ be a poset and \textup{$\mathcal{F}$ a conical system on
$P(Y)$. For $m\geq0$ and $h\in S_{m}(Y)$, we have
\[
\partial_{m}T_{1}(\mathcal{F},h)=\sum_{j=1}^{m}(-1)^{j+1}T_{1}(\mathcal{F},h\circ d^{j}).
\]
}\end{lem}
\begin{proof}
By the definition of $T_{1}(\mathcal{F},h)$, we have
\begin{align*}
\partial_{m}T_{1}(\mathcal{F},h) & =\partial_{m}\sum_{\sigma\in\mathfrak{S}_{m}}\mathrm{sgn}(\sigma)\mathcal{F}(g_{\sigma,h})\\
 & =\sum_{\sigma\in\mathfrak{S}_{m}}\sum_{j=1}^{m}(-1)^{j+1}\mathrm{sgn}(\sigma)\mathcal{F}(g_{\sigma,h}\circ d^{j}).
\end{align*}
Since $\sum_{\sigma\in\mathfrak{S}_{m}}\mathrm{sgn}(\sigma)\mathcal{F}(g_{\sigma,h}\circ d^{j})=0$
for all $1\leq j\leq m-1$, we have
\begin{align*}
\sum_{\sigma\in\mathfrak{S}_{m}}\sum_{j=1}^{m}(-1)^{j+1}\mathrm{sgn}(\sigma)\mathcal{F}(g_{\sigma,h}\circ d^{j}) & =\sum_{\sigma\in\mathfrak{S}_{m}}\mathrm{sgn}(\sigma)(-1)^{m+1}\mathcal{F}(g_{\sigma,h}\circ d^{m}).\\
 & =\sum_{i=1}^{n}\sum_{\sigma\in\mathfrak{S}_{m-1}}\mathrm{sgn}(\sigma)(-1)^{i+1}\mathcal{F}(g_{u(m,i,\sigma),h}\circ d^{m}).
\end{align*}
We have 
\[
g_{u(m,i,\sigma),h}\circ d^{m}=g_{\sigma,h\circ d^{i}}
\]
 since 
\begin{align*}
(g_{u(m,i,\sigma),h}\circ d^{m})(j) & =\{h(d^{i}(\sigma(k)))\mid1\leq k\leq j\}\\
 & =g_{\sigma,h\circ d_{i}}(j)
\end{align*}
for all $1\leq j\leq m-1$. Thus we have
\begin{align*}
\sum_{\sigma\in\mathfrak{S}_{m-1}}\mathrm{sgn}(\sigma)\mathcal{F}(g_{u(m,i,\sigma),h}\circ d^{m}) & =\sum_{\sigma\in\mathfrak{S}_{m-1}}\mathrm{sgn}(\sigma)\mathcal{F}(g_{\sigma,h\circ d^{i}})\\
 & =T_{1}(\mathcal{F},h\circ d^{i}).
\end{align*}
It follows that 
\[
\partial_{m}T_{1}(\mathcal{F},h)=\sum_{j=1}^{m}(-1)^{j+1}T_{1}(\mathcal{F},h\circ d^{j})
\]
and the lemma is proved.\end{proof}
\begin{cor}
\label{cor:T1_conic_sytem}Let $Y$ be a poset and\textup{ $\mathcal{F}$
a conical system on $P(Y)$. We define $\mathcal{G}:\bigcup_{k=0}^{\infty}S_{k}(Y)\to\bigcup_{k=0}^{\infty}C_{k}(X)$
by
\[
\mathcal{G}(h)=T_{1}(\mathcal{F},h).
\]
Then $\mathcal{G}$ is a conical system on $Y$. }
\end{cor}
This conical system plays a fundamental role in Section \ref{sub:Dual-fans}
and \ref{sub:Dual-fan-of-fundamental-domain-fan}.

\subsection{Cubic fundamental domain fans}

Let $G$ be an abelian group which acts on $X$. Then $G$ acts on
$C_{n}(X)$ by $gx=\phi_{g}(x)$ for $(g,x)\in G\times C_{n}(X)$,
where $\phi_{g}:C_{n}(X)\to C_{n}(X)$ is a homomorphism defined by
\[
\phi_{g}(\Lambda(a_{1},\dots,a_{n}))=\Lambda(ga_{1},\dots,ga_{n})
\]
for $a_{1},\dots,a_{n}\in X$. We define $I_{m}(X,G)\subset C_{m}(X)$
as the $\mathbb{Z}$-submodule generated by
\[
gx-x\;\;\;(g\in G,\ x\in C_{m}(X)).
\]
Let $B(n)=\{0,1\}^{n}\subset\mathbb{Z}^{n}$. We define the partial
order on $B(n)$ by letting
\[
(b_{1},\dots,b_{n})\leq(b_{1}',\dots,b_{n}')
\]
if and only if $b_{j}\leq b_{j}'$ for all $1\leq j\leq n$. Let $e_{1},\dots,e_{n}$
be the elements of $B(n)$ defined by $e_{1}=(1,0,\dots,0),\dots,e_{n}=(0,\dots,0,1).$
Let $A_{n}$ be the set of all conical systems on $B(n)$. For $\mathcal{F}\in A_{n}$,
we define $T(\mathcal{F})\in C_{n+1}(X)$ by
\[
T(\mathcal{F})=\sum_{\sigma\in\mathfrak{S}_{n}}\mathrm{sgn}(\sigma)\mathcal{F}(f_{\sigma})
\]
where $f_{\sigma}\in S_{n+1}(B(n))$ is defined by
\[
f_{\sigma}(j)=\sum_{k=1}^{j-1}e_{\sigma(k)}
\]
for $1\leq j\leq n+1.$
\begin{example}
In the case $n=3$, $T(\mathcal{F})\in C_{4}(X)$ is given as follows:
\begin{align*}
T(\mathcal{F}) & =\mathcal{F}(\left[(0,0,0),(1,0,0),(1,1,0),(1,1,1)\right])+\mathcal{F}(\left[(0,0,0),(0,1,0),(0,1,1),(1,1,1)\right])\\
 & \ \ \ +\mathcal{F}(\left[(0,0,0),(0,0,1),(1,0,1),(1,1,1)\right])-\mathcal{F}(\left[(0,0,0),(1,0,0),(1,0,1),(1,1,1)\right])\\
 & \ \ \ -\mathcal{F}(\left[(0,0,0),(0,1,0),(1,1,0),(1,1,1)\right])-\mathcal{F}(\left[(0,0,0),(0,0,1),(0,1,1),(1,1,1)\right]).
\end{align*}

\end{example}
Let us calculate the boundary of $T(\mathcal{F})$. For $b\in\{0,1\}$
and $1\leq i\leq n$, we define $w_{b}^{i}:B(n-1)\to B(n)$ and $\alpha_{b}^{i}:A_{n}\to A_{n-1}$
by 
\[
w_{b}^{i}((b_{1},\dots,b_{n-1}))=(b_{1},\dots,b_{i},b,b_{i+1},\dots,b_{n-1})
\]
and $(\alpha_{b}^{i}\mathcal{F})(f)=\mathcal{F}(w_{b}^{i}\circ f)$.
The boundary of $T(\mathcal{F})$ is given by the following lemma.
\begin{lem}
\label{lem:boundary_op_top}For all $\mathcal{F}\in A_{n}$, we have
\[
\partial_{n+1}T(\mathcal{F})=\sum_{i=1}^{n}(-1)^{i+1}\left(T(\alpha_{1}^{i}\mathcal{F})-T(\alpha_{0}^{i}\mathcal{F})\right).
\]
\end{lem}
\begin{proof}
By the definition of $T(\mathcal{F})$ and $A_{n}$, we have
\begin{align*}
\partial_{n+1}T(\mathcal{F}) & =\partial_{n+1}\sum_{\sigma\in\mathfrak{S}_{n}}\mathrm{sgn}(\sigma)\mathcal{F}(f_{\sigma})\\
 & =\sum_{\sigma\in\mathfrak{S}_{n}}\sum_{j=1}^{n+1}(-1)^{j+1}\mathrm{sgn}(\sigma)\mathcal{F}(f_{\sigma}\circ d^{j}).
\end{align*}
Since $\sum_{\sigma\in\mathfrak{S}_{n}}\mathrm{sgn}(\sigma)\mathcal{F}(f_{\sigma}\circ d^{j})=0$
for all $2\leq j\leq n$, we have
\[
\sum_{\sigma\in\mathfrak{S}_{n}}\sum_{j=1}^{n+1}(-1)^{j+1}\mathrm{sgn}(\sigma)\mathcal{F}(f_{\sigma}\circ d^{j})=\sum_{\sigma\in\mathfrak{S}_{n}}\mathrm{sgn}(\sigma)\left(\mathcal{F}(f_{\sigma}\circ d^{1})+(-1)^{n}\mathcal{F}(f_{\sigma}\circ d^{n+1})\right).
\]
We have
\[
f_{u(1,i,\sigma)}\circ d^{1}=w_{1}^{i}\circ f_{\sigma}
\]
since 
\begin{align*}
(f_{u(1,i,\sigma)}\circ d^{1})(j) & =f_{u(1,i,\sigma)}(j+1)\\
 & =\sum_{k=1}^{j}e_{u(1,i,\sigma)(k)}\\
 & =e_{i}+\sum_{k=2}^{j}e_{d^{i}(\sigma(k-1))}\\
 & =w_{1}^{i}(\sum_{k=1}^{j-1}e_{\sigma(k)})\\
 & =w_{1}^{i}(f_{\sigma}(j)).
\end{align*}
Thus we have
\begin{align*}
\sum_{\sigma\in\mathfrak{S}_{n}}\mathrm{sgn}(\sigma)\mathcal{F}(f_{\sigma}\circ d^{1}) & =\sum_{i=1}^{n}\sum_{\sigma\in\mathfrak{S}_{n-1}}\mathrm{sgn}(u(1,i,\sigma))\mathcal{F}(f_{u(1,i,\sigma)}\circ d^{1})\\
 & =\sum_{i=1}^{n}\sum_{\sigma\in\mathfrak{S}_{n-1}}(-1)^{i+1}\mathrm{sgn}(\sigma)\mathcal{F}(w_{1}^{i}\circ f_{\sigma})\\
 & =\sum_{i=1}^{n}\sum_{\sigma\in\mathfrak{S}_{n-1}}(-1)^{i+1}\mathrm{sgn}(\sigma)(\alpha_{1}^{i}\mathcal{F})(f_{\sigma})\\
 & =\sum_{i=1}^{n}(-1)^{i+1}T(\alpha_{1}^{i}\mathcal{F}).
\end{align*}
Similarly we have
\begin{align*}
\sum_{\sigma\in\mathfrak{S}_{n}}\mathrm{sgn}(\sigma)\mathcal{F}(f_{\sigma}\circ d^{n+1}) & =\sum_{i=1}^{n}\sum_{\sigma\in\mathfrak{S}_{n-1}}\mathrm{sgn}(u(n,i,\sigma))\mathcal{F}(f_{u(n,i,\sigma)}\circ d^{n+1})\\
 & =\sum_{i=1}^{n}\sum_{\sigma\in\mathfrak{S}_{n-1}}(-1)^{n+i}\mathrm{sgn}(\sigma)\mathcal{F}(w_{0}^{i}\circ f_{\sigma})\\
 & =\sum_{i=1}^{n}(-1)^{n+i}T(\alpha_{0}^{i}\mathcal{F}).
\end{align*}
Hence the lemma is proved.\end{proof}
\begin{defn}
For $g_{1},\dots,g_{n}\in G$, we define $A(g_{1},\dots,g_{n})\subset A_{n}$
to be the set of conical systems $\mathcal{F}$ on $B(n)$ which satisfy
the following condition:

Let $1\leq i\leq n$ and $f_{1},f_{2}\in S_{k}(B(n))$. If $f_{1}(j)=e_{i}+f_{2}(j)$
for all $1\leq j\leq k$ then $\mathcal{F}(f_{1})=g_{i}\mathcal{F}(f_{2})$.\end{defn}
\begin{example}
Let $x\in X$ and $g_{1},\dots,g_{n}\in G$. For $e=(b_{1},\dots b_{n})\in B(n)$,
we define $E(e)\in X$ by 
\[
E(e)=g_{1}^{b_{1}}\cdots g_{n}^{b_{n}}x.
\]
Then the standard conical system $\mathcal{S}_{E}$ attached to $E$
satisfies $\mathcal{S}_{E}\in A(g_{1},\dots,g_{n})$. We define the
\emph{cubic fundamental domain fans} $F(x;g_{1},\dots,g_{n})\in C_{n+1}(X)$
by
\begin{align*}
F(x;g_{1},\dots,g_{n}) & =T(\mathcal{S}_{E})\\
 & =\sum_{\sigma\in S_{n}}\mathrm{sgn}(\sigma)\Lambda(x,g_{\sigma(1)}x,g_{\sigma(1)}g_{\sigma(2)}x,\dots,g_{\sigma(1)}g_{\sigma(2)}\cdots g_{\sigma(n)}x).
\end{align*}

\end{example}
The cubic fundamental domain fans $F(x;g_{1},\dots,g_{n})$ play a
fundamental role in Theorem \ref{Fact:exists-epsilon}.
\begin{lem}
\label{lem:equivalence-of-T}For $g_{1},\dots,g_{n}\in G$ and $\mathcal{F},\mathcal{G}\in A(g_{1},\dots,g_{n})$,
we have
\[
T(\mathcal{F})-T(\mathcal{G})\in I_{n+1}(X,G)+\ker\partial_{n+1}.
\]
\end{lem}
\begin{proof}
We define $\mathcal{H}\in A_{n+1}$ by
\[
\mathcal{H}(f)=\mathcal{F}(f_{1})\otimes\mathcal{G}(f_{2})
\]
where $f\in S_{k}(B(n+1))$, $f_{1}\in S_{k_{1}}(B(n))$, $f_{2}\in S_{k_{2}}(B(n))$,
$k=k_{1}+k_{2}$ and 
\[
f(j)=\begin{cases}
w_{0}^{n+1}(f_{1}(j)) & 1\leq j\leq k_{1}\\
w_{1}^{n+1}(f_{2}(j-k_{1})) & k_{1}+1\leq j\leq k_{1}+k_{2}.
\end{cases}
\]
It is obvious that $\mathcal{H}\in A_{n+1}$ from the definition.
By Lemma \ref{lem:boundary_op_top}, we have
\[
\partial_{n+2}T(\mathcal{H})=\sum_{i=1}^{n+1}(-1)^{i+1}\left(T(\alpha_{1}^{i}\mathcal{H})-T(\alpha_{0}^{i}\mathcal{H})\right)\in\ker\partial_{n+1}.
\]
For $1\leq i\leq n$, we have
\[
T(\alpha_{1}^{i}\mathcal{H})=g_{i}T(\alpha_{0}^{i}\mathcal{H})
\]
since 
\begin{align*}
(\alpha_{1}^{i}\mathcal{H})(f) & =\mathcal{H}(w_{1}^{i}\circ f)\\
 & =g_{i}\mathcal{H}(w_{0}^{i}\circ f)\\
 & =g_{i}(\alpha_{0}^{i}\mathcal{H})(f).
\end{align*}
Therefore 
\[
T(\alpha_{1}^{i}\mathcal{H})-T(\alpha_{0}^{i}\mathcal{H})\in I_{n+1}(X,G)
\]
 for $1\leq i\leq n$. We have $T(\alpha_{1}^{n+1}\mathcal{H})=T(\mathcal{G})$
since
\begin{align*}
(\alpha_{1}^{n+1}\mathcal{H})(f) & =\mathcal{H}(w_{1}^{n+1}\circ f)\\
 & =\mathcal{F}(\emptyset)\otimes\mathcal{G}(f)\\
 & =\mathcal{G}(f).
\end{align*}
Similarly we have $T(\alpha_{0}^{n+1}\mathcal{H})=T(\mathcal{F})$.
Hence the lemma is proved.
\end{proof}
The following proposition immediately follows from Lemma \ref{lem:equivalence-of-T}.
\begin{prop}
\label{prop:another_cubic_fandamental_domain}For any $x,y\in X$,
we have 
\[
F(y;g_{1},\dots,g_{n-1})-F(x;g_{1},\dots,g_{n-1})\in I_{n}(X,G)+\ker\partial_{n}.
\]

\end{prop}
The following lemma follows from the standard argument in homological
algebra.
\begin{lem}
\label{lem:another_basis}Let $g_{1},\dots,g_{n}\in G$ and $C=(c_{ij})_{i,j=1}^{n}\in M_{n}(\mathbb{Z})$.
Put $g_{i}'=\prod_{j=1}^{n}g_{j}^{c_{ij}}$ for all $1\leq i\leq n$.
Then
\[
F(x;g_{1}',\dots,g_{n}')-\det(C)F(x;g_{1},\dots,g_{n})\in I_{n+1}(X,G)+\ker\partial_{n+1}
\]
for all $x\in X$. \end{lem}
\begin{proof}
Let $E$ be a free abelian group generated by $e_{1},\dots,e_{n}\in E$.
We define a group homomorphism $s:E\to G$ by $s(e_{1})=g_{1},\dots,s(e_{n})=g_{n}$.
We set $e_{i}'=\prod_{j=1}^{n}e_{j}^{c_{ij}}$ for all $1\leq i\leq n$.
Note that $s(e_{i}')=g_{i}'$ for all $1\leq i\leq n$. We define
a homomorphism $t:C_{m}(E)\to C_{m}(X)$ by 
\[
t(\Lambda(u_{1},\dots,u_{n}))=\Lambda(s(u_{1})x,\dots,s(u_{n})x)
\]
for $u_{1},\dots,u_{n}\in E$. From the definition of group homology,
we have
\[
\partial_{n+1}^{-1}(I_{n}(E,E))/(I_{n+1}(E,E)+\ker\partial_{n+1})=H_{n}(E,\mathbb{Z}).
\]
Therefore there exists a natural isomorphism 
\[
\eta:\partial_{n+1}^{-1}(I_{n}(E,E))/(I_{n+1}(E,E)+\ker\partial_{n+1})\to\wedge^{n}E.
\]
We have
\[
\eta(F(1;u_{1},\dots,u_{n}))=u_{1}\wedge\cdots\wedge u_{n}
\]
for all $u_{1},\dots,u_{n}\in E$. Thus we have
\[
\eta(F(1;e_{1}',\dots,e_{n}'))=\det(C)\eta(F(1;e_{1},\dots,e_{n})).
\]
Therefore we have
\begin{equation}
F(1;e_{1}',\dots,e_{n}')-\det(C)F(1;e_{1},\dots,e_{n})\in I_{n+1}(E,E)+\ker\partial_{n+1}.\label{eq:00}
\end{equation}
By applying $t$ to (\ref{eq:00}), we have
\[
F(x;g_{1}',\dots,g_{n}')-\det(C)F(x;g_{1},\dots,g_{n})\in I_{n+1}(X,G)+\ker\partial_{n+1}.
\]

\end{proof}

\subsection{Dual fans\label{sub:Dual-fans}}

Let $V$ be an $n$-dimensional vector space over a field $L$. We
consider $C_{n}(V\setminus\{0\})$. For simplicity, we use the notation
$C_{n}(V)$ instead of $C_{n}(V\setminus\{0\})$. We define $W(V)\subset C_{n}(V)$
as the submodule generated by
\[
\{\Lambda(a_{1},\dots,a_{n})\mid a_{1},\dots,a_{n}\in V\setminus\{0\},\ a_{1},\dots,a_{n}\mbox{ are linearly dependent}\}.
\]
We define the duality map.
\begin{defn}
The duality map $\varphi:C_{n}(V)\rightarrow C_{n}(V^{*})$ is a homomorphism
defined by
\[
\varphi(\Lambda(a_{1},\cdots,a_{n}))=\begin{cases}
\Lambda(b_{1},\dots,b_{n}) & a_{1},\dots,a_{n}:\mbox{linearly independent}\\
0 & a_{1},\dots,a_{n}:\mbox{linearly dependent}
\end{cases}
\]
for $a_{1},\dots,a_{n}\in V\setminus\{0\}$, where $b_{1},\dots,b_{n}$
is the dual basis of $a_{1},\dots,a_{n}$.
\end{defn}
We call $\varphi(\mathbb{B})$ the \emph{dual fan} of $\mathbb{B}$.
Let $Y$ be a poset and $c:Y\to V\setminus\{0\}$ a map. We define
$D(c)$ to be the set consisting of all maps 
\[
\mathcal{D}:P(Y)\to V^{*}\setminus\{0\}
\]
 such that $\left\langle c(a),\mathcal{D}(U)\right\rangle =0$ for
all $U\in P(Y)$ and $a\in U$ such that $\#U<n$. There is no restriction
on the choice of $\mathcal{D}(U)$ for $U\subset A$ such that $\#U\geq n$.
From the definition, it is obvious that $D(c)\neq\emptyset$ for all
$c$. 
\begin{lem}
\label{lem:constant-multiply}For $t_{1},\dots,t_{n}\in L^{\times}$
and $a_{1},\dots,a_{n}\in V\setminus\{0\}$, we have
\[
\Lambda(a_{1},\dots,a_{n})-\Lambda(t_{1}a_{1},\dots,t_{n}a_{n})\in W(V)+\ker\partial_{n}
\]
\end{lem}
\begin{proof}
It is enough to prove that

\[
x=\Lambda(a_{1},\dots a_{k-1},ta_{k},a_{k+1},\dots,a_{n})-\Lambda(a_{1},\dots,a_{k-1},a_{k},a_{k+1},\dots,a_{n})\in W(V)+\ker\partial_{n}
\]
which follows from 
\[
\partial_{n+1}\Lambda(a_{1},\dots a_{k-1},ta_{k},a_{k},\dots,a_{n})-(-1)^{k}x\in W(V).
\]

\end{proof}
The next lemma immediately follows from Lemma \ref{lem:constant-multiply}.
\begin{lem}
Let $Y$ be a set, $c:Y\to V\setminus\{0\}$ a map, and $\mathcal{D}\in D(c)$
a map. Then we have
\[
\varphi\Lambda(c(a_{1}),\dots,c(a_{n}))-\Lambda(\mathcal{D}(\{a_{2},\dots,a_{n}\}),\dots,\mathcal{D}(\{a_{1},\dots,a_{n-1}\}))\in W(V^{*})+\ker\partial_{n}
\]
for all $a_{1}<\cdots<a_{n}\in Y$.
\end{lem}
Let $Y$ be a poset and $c:Y\to V\setminus\{0\}$ a map. Let $\mathcal{S}_{c}$
be the standard conical system on $Y$ attached to $c$. We want to
construct a natural conical system $\mathcal{F}$ on $Y$ with values
in $\bigcup_{m}C_{m}(V^{*})$ such that 
\[
\mathcal{F}(h)+(-1)^{n}\varphi(\mathcal{S}_{c}(h))\in W(V^{*})+\ker\partial_{n}
\]
for all $h\in S_{n}(Y)$. The following lemma gives such a conical
system.
\begin{lem}
\label{lem:dual_equation}Let $\mathcal{D}\in D(c)$. Let $\mathcal{F}$
be a conical system on $Y$ defined by $\mathcal{F}(h)=T_{1}(\mathcal{S}_{\mathcal{D}},h)$
where $\mathcal{S}_{\mathcal{D}}$ is the standard conical system
attached to $\mathcal{D}$. Then we have 
\[
\mathcal{F}(h)+(-1)^{n}\varphi(\mathcal{S}_{c}(h))\in W(V^{*})+\ker\partial_{n}
\]
for all $h\in S_{n}(Y)$.\end{lem}
\begin{proof}
It is enough to prove in the case where $Y=\{1,\dots,n\}$ and $h$
is the identity function. Let $\mathcal{S}_{d}$ be the standard conical
system attached to $d$ where $d:\{1,\dots n\}\to V^{*}\setminus\{0\}$
is defined by
\[
d(j)=\mathcal{D}(\{1,\dots,n\}\setminus\{j\}).
\]
Then by Lemma \ref{lem:fusion_product}, we have
\[
\sum_{k=0}^{n}\sum_{\substack{f\in S_{k}(\{1,\dots,n\})\\
g\in S_{n-k}(\{1,\dots,n\})\\
\mathrm{Im}f\cap\mathrm{Im}g=\emptyset
}
}(-1)^{k}\mathrm{sgn}(f,g)\mathcal{F}(f)\otimes\mathcal{S}_{d}(g)\in\ker\partial_{n}.
\]
Let $1\leq k\leq n-1$, $f\in S_{k}(\{1,\dots,n\})$ and $g\in S_{n-k}(\{1,\dots,n\}).$
Assume that $\mathrm{Im}f\cap\mathrm{Im}g=\emptyset$. Then we have
\[
\mathcal{F}(f)\otimes\mathcal{S}_{d}(g)=\sum_{\sigma\in\mathfrak{S}_{k}}\Lambda(a_{\sigma,1},\dots,a_{\sigma,k},d(g(1)),\dots d(g(n-k)))
\]
where $a_{\sigma,i}\in V^{*}\setminus\{0\}$ for $1\leq i\leq k$
is define by 
\[
a_{\sigma,i}=\mathcal{D}(\{f(\sigma(j))\mid1\leq j\leq i\}).
\]
We have 
\[
\Lambda(a_{\sigma,1},\dots,a_{\sigma,k},d(g(1)),\dots d(g(n-k)))\in W(V^{*})
\]
 for all $\sigma\in\mathfrak{S}_{k}$ since 
\[
\left\langle c(f(\sigma(1))),v\right\rangle =0
\]
 for all 
\[
v\in\{a_{\sigma,1},\dots,a_{\sigma,k},d(g(1)),\dots d(g(n-k))\}.
\]
 Therefore we have
\[
\sum_{k=1}^{n-1}\sum_{\substack{f\in S_{k}(\{1,\dots,n\})\\
g\in S_{n-k}(\{1,\dots,n\})\\
\mathrm{Im}f\cap\mathrm{Im}g=\emptyset
}
}(-1)^{k}\mathrm{sgn}(f,g)\mathcal{F}(f)\otimes\mathcal{S}_{d}(g)\in W.
\]
Thus we have
\[
\mathcal{F}(\emptyset)\otimes\mathcal{S}_{d}(\mathrm{id}_{n})+(-1)^{n}\mathcal{F}(\mathrm{id}_{n})\otimes\mathcal{S}_{d}(\emptyset)\in W(V^{*})+\ker\partial_{n}.
\]
Since we have 
\[
\mathcal{F}(\emptyset)\otimes\mathcal{S}_{d}(\mathrm{id}_{n})-\varphi\Lambda(c(1),\dots,c(n))\in W(V^{*})
\]
 and 
\[
\mathcal{F}(\mathrm{id}_{n})\otimes\mathcal{S}_{d}(\emptyset)=T_{1}(\mathcal{S}_{\mathcal{D}},\mathrm{id}_{n}),
\]
the lemma is proved.
\end{proof}

\begin{prop}
\label{prop:duality_map_factor}For $x\in W(V)+\ker\partial_{n}$,
we have $\varphi(x)\in W(V^{*})+\ker\partial_{n}$.\end{prop}
\begin{proof}
It is enough to prove that

\[
\varphi(\sum_{i=1}^{n+1}(-1)^{i+1}\Lambda(a_{1},\dots,\widehat{a_{i}},\dots,a_{n+1}))\in W(V^{*})+\ker\partial_{n}.
\]
We define $c:\{1,\dots,n+1\}\to V\setminus\{0\}$ by $c(i)=a_{i}.$
Take any element $\mathcal{D}\in D(c).$ By Lemma \ref{lem:dual_equation},
we have
\[
\varphi(\Lambda(a_{1},\dots,\widehat{a_{i}},\dots,a_{n+1}))+(-1)^{n}T_{1}(\mathcal{S}_{\mathcal{D}},d^{i}\circ{\rm id}_{n})\in W(V^{*})+\ker\partial_{n}
\]
for all $1\leq i\leq n+1$. Therefore it is enough to prove that 
\[
\sum_{i=1}^{n+1}(-1)^{i+1}T_{1}(\mathcal{S}_{\mathcal{D}},d^{i}\circ{\rm id}_{n})\in W(V^{*})+\ker\partial_{n}.
\]
Since
\[
\partial_{n+1}T_{1}(\mathcal{S}_{\mathcal{D}},{\rm id}_{n+1})=\sum_{i=1}^{n+1}(-1)^{i+1}T_{1}(\mathcal{S}_{\mathcal{D}},d^{i}\circ{\rm id}_{n}),
\]
the lemma is proved.
\end{proof}
By Proposition \ref{prop:duality_map_factor}, we can define 
\[
\varphi:C_{n}(V)/(W(V)+\ker\partial_{n})\rightarrow C_{n}(V^{*})/(W(V^{*})+\ker\partial_{n}).
\]
It is trivial that $\varphi\circ\varphi=\mathrm{id}$.

\subsection{\label{sub:Dual-fan-of-fundamental-domain-fan}The dual fan of a
fundamental domain fan.}

Let $V$ be an $n$-dimensional vector space over a field $L$, and
let $G$ be an abelian group which acts on $V\setminus\{0\}$ linearly.
We define the action of $G$ on $V^{*}$ by $\left\langle x,gy\right\rangle =\left\langle gx,y\right\rangle $
for $g\in G$, $x\in V$, and $y\in V^{*}$. In Section \ref{sec:Construction-of-Hecke},
we will consider the case where $L=\mathbb{Q}$, $V=K$ and $G=E_{K}$.
We identify $K$ and $K^{*}$ by the inner product $\left\langle x,y\right\rangle =\mathrm{Tr}(xy).$

We define the notion of a\emph{ }fundamental domain fan for $g_{1},\dots,g_{n-1}\in G$.
\begin{defn}
For $a\in V\setminus\{0\}$ and $g_{1},\dots,g_{n-1}\in G$, we call
$\mathbb{D}\in C_{n}(V)$ a\emph{ fundamental domain fan} for $(g_{1},\dots,g_{n-1})$
if
\[
\mathbb{D}-F(a;g_{1},\dots,g_{n-1})\in I_{n}(V^{*}\setminus\{0\},G)+W(V^{*})+\ker\partial_{n}.
\]
This condition does not depend on the choice of $a\in V\setminus\{0\}$.\end{defn}
\begin{thm}
\label{thm:dual_of_fandamental_domain_is_fandamental_domain}Let $\mathbb{D}\in C_{n}(V)$
be a fundamental domain fan for $(g_{1},\dots,g_{n-1})$. Then $\varphi(\mathbb{D})\in C_{n}(V^{*})$
is a fundamental domain fan for $(g_{1},\dots,g_{n-1})$.\end{thm}
\begin{proof}
It is enough to prove that 
\[
\varphi(F(a;g_{1},\dots,g_{n-1}))-F(b;g_{1},\dots,g_{n-1})\in I_{n}(V^{*}\setminus\{0\},G)+W(V^{*})+\ker\partial_{n}
\]
for $a\in V\setminus\{0\}$ and $b\in V^{*}\setminus\{0\}$. We define
$c:B(n-1)\to V\setminus\{0\}$ by
\[
c((b_{1},\dots,b_{n-1}))=g_{1}^{b_{1}}\cdots g_{n-1}^{b_{n-1}}a.
\]
Let
\[
\mathcal{D}:P(B(n-1))\to V^{*}\setminus\{0\}
\]
be any map which satisfies the following conditions:
\begin{enumerate}
\item \label{enu:1}$\mathcal{D}\in D(c).$
\item \label{enu:2}If $U_{1},U_{2}\in P(B(n-1))$ satisfy 
\[
U_{1}=\{z+e_{i}\mid z\in U_{2}\}\ \ \ \ \ (1\leq i\leq n-1)
\]
then $\mathcal{D}(U_{1})=g_{i}^{-1}\mathcal{D}(U_{2})$.
\end{enumerate}
The existence of such a map is obvious (note that $\emptyset\notin P(B(n-1))$).
We define the conical system $\mathcal{F}$ on $B(n-1)$ with values
in $\bigcup_{m}C_{m}(V^{*})$ by
\[
\mathcal{F}(f)=T_{1}(\mathcal{S}_{\mathcal{D}},f)
\]
where $f\in S_{k}(B(n-1))$ and $0\leq k\leq n.$ Then we have $\mathcal{F}\in A(g_{1}^{-1},\dots,g_{n}^{-1})$.
By Lemma \ref{lem:dual_equation}, we have
\begin{align*}
\varphi F(a;g_{1},\dots,g_{n-1}) & =\sum_{\sigma\in\mathfrak{S}_{n-1}}\mathrm{sgn}(\sigma)\varphi\Lambda(c(f_{\sigma}(1)),c(f_{\sigma}(2)),\dots,c(f_{\sigma}(n)))\\
 & =(-1)^{n-1}\sum_{\sigma\in\mathfrak{S}_{n-1}}\mathrm{sgn}(\sigma)T_{1}(\mathcal{S}_{\mathcal{D}},f_{\sigma})\\
 & =(-1)^{n-1}\sum_{\sigma\in\mathfrak{S}_{n-1}}\mathrm{sgn}(\sigma)\mathcal{F}(f_{\sigma})\\
 & =(-1)^{n-1}T(\mathcal{F})
\end{align*}
in $C_{n}/(W(V^{*})+\ker\partial_{n})$, where $f_{\sigma}\in S_{n}(B(n-1))$
is defined by
\[
f_{\sigma}(j)=\sum_{k=1}^{j-1}e_{\sigma(k)}.
\]
By Lemma \ref{lem:equivalence-of-T}, we have
\[
T(\mathcal{F})-F(b;g_{1}^{-1},\dots,g_{n-1}^{-1})\in I_{n}(V^{*}\setminus\{0\},G)+\ker\partial_{n}.
\]
By Lemma \ref{lem:another_basis} we have 
\[
F(b;g_{1}^{-1},\dots,g_{n-1}^{-1})-(-1)^{n-1}F(b;g_{1},\dots,g_{n-1})\in I_{n}(V^{*}\setminus\{0\},G)+\ker\partial_{n}.
\]
Hence we have 
\[
\varphi F(a;g_{1},\dots,g_{n-1})-F(b;g_{1},\dots,g_{n-1})\in I_{n}(V^{*}\setminus\{0\},G)+W(V^{*})+\ker\partial_{n},
\]
and the theorem is proved.
\end{proof}

\subsection{\label{sub:Indicator-function-of-fan}The property of the characteristic
function of a fan}

Let $V$ be an $n$-dimensional vector space over $\mathbb{Q}$. We
say that a set $X\subset V$ is thin if it is contained in a finite
union of proper vector subspaces of $V$. Let $(\leq)$ be a total
order on the set $V$ such that $(V,\leq)$ is an ordered vector space.
For $v\in V$, we say that $v$ is positive if $0<v$. For a nonzero
vector $v\in V$, we define $P(v)$ by
\[
P(v)=\begin{cases}
v & v\mbox{ is positive}\\
-v & \mbox{otherwise}.
\end{cases}
\]
Note that $P(v)$ is always positive if $v\neq0$. For a cone $\Lambda\in C_{r}(V)$,
we define $P(\Lambda)\in C_{r}(V)$ by
\[
P(\Lambda(u_{1},\dots,u_{r}))=\Lambda(P(u_{1}),\dots,P(u_{r})).
\]
For a fan $\mathbb{B}=\sum_{j}c_{j}\Lambda_{j}\in C_{r}(V)$, we define
$P(\mathbb{B})\in C_{r}(V)$ by $P(\mathbb{B})=\sum_{j}c_{j}P(\Lambda_{j})$.
Then we have the following lemma.
\begin{lem}
\label{lem:W-Ker-shrink}Let $\mathbb{B}\in W(V)+\ker\partial_{n}$
be a fan. Then $\mathfrak{C}(P(\mathbb{B}),r)$ has a thin support.\end{lem}
\begin{proof}
It is enough to prove in the case where $\mathbb{B}=\partial_{n+1}\Lambda(u_{0},u_{1},\dots,u_{n})$
with $u_{0},\dots,u_{n}\in V\setminus\{0\}$. We may assume that $u_{0},\cdots,u_{n}$
are positive by changing $u_{j}$ to $P(u_{j})$. If $u_{0},\dots,u_{n}$
is not a generator of $V$ then the lemma is obvious. Thus we may
assume that $u_{0},\dots,u_{n}$ is a generator of $V$. Further,
we may assume that $u_{1},\dots,u_{n}$ are linearly independent by
permutating $u_{0},\dots,u_{n}$, and that $u_{0}=a_{1}u_{1}+\cdots+a_{k}u_{k}$
for some $1\leq k\leq n$ and $a_{1},\dots,a_{k}\in\mathbb{Q}\setminus\{0\}$
by permutating $u_{1},\dots,u_{n}$. We denote by $V_{k}$ the subspace
of $V$ generated by $u_{1},\dots,u_{k}$. We set $\mathbb{B}_{k}=\partial_{k+1}\Lambda(u_{0},u_{1},\dots,u_{k})\in C_{k}(V_{k})$
and $\mathbb{B}'=\sum_{j=0}^{k}(-1)^{j}\Lambda(u_{0},\dots,u_{j-1},u_{j+1},\dots,u_{n})$.
Since $\mathbb{B}-\mathbb{B}'\in W(V)$, it is enough to prove that
$\mathfrak{C}(\mathbb{B}')$ has a thin support. By the definition
we have
\[
\mathfrak{C}(\mathbb{B}')(w+t_{k+1}u_{k+1}+\cdots+t_{n}u_{n})=\begin{cases}
\mathfrak{C}(\mathbb{B}_{k})(w) & t_{j}>0\mbox{ for all }k<j\leq n\\
0 & \mbox{otherwise}
\end{cases}
\]
where $w\in V_{k}$ and $t_{k+1},\dots,t_{n}\in\mathbb{R}^{n}$. Since
$\mathfrak{C}(\mathbb{B}_{k}):V_{K}\to\mathbb{Z}$ has a thin support
by Proposition 2 in \cite{MR2392823}, $\mathfrak{C}(\mathbb{B}')$
has a thin support. Thus the lemma is proved.
\end{proof}

\section{\label{sec:Construction-of-Hecke}Expression of Hecke L-functions
of totally real fields}

\subsection{Hecke characters}

Let $K$ be a totally real field and $\chi=\prod_{v}\chi_{v}$ a Hecke
character of $K$. Let $\mathfrak{m}$ be the conductor of $\chi$,
and let $\mathfrak{d}$ be the different of $K$. For a place $v$
of $K$, $K_{v}$ will denote the completion of $K$ at $v$. We set
$\chi_{\infty}=\prod_{v}\chi_{v}$ where $v$ runs through all infinite
places. For a finite place $v$, $r_{v}$ will denote the maximal
compact subring of $K_{v}$, and $\pi_{v}$ will denote a prime element
of $r_{v}$. Let $\mathfrak{p}_{v}$ be the prime ideal of $O_{K}$
corresponding to a finite place $v$. For a fractional ideal $\mathfrak{a}=\prod_{v}\mathfrak{p}_{v}^{f_{v}}$
of $K$, we define $I(\mathfrak{a})\subset K(\mathbb{A}_{f})^{\times}$
by 
\[
I(\mathfrak{a})=\{(a_{v})_{v}\in K(\mathbb{A}_{f})^{\times}\mid a_{v}\in\pi_{v}^{f_{v}}r_{v}^{\times}\mbox{ for all finite places of }K\}.
\]
We define the standard function $\Phi_{\chi}\in\mathcal{S}(K(\mathbb{A}_{f}))$
attached to $\chi$ by $\Phi_{\chi}=\prod_{v}\Phi_{v}$ where $\Phi_{v}\in\mathcal{S}(K_{v})$
is defined by
\[
\Phi_{v}(x)=\begin{cases}
1 & x\in r_{v}\\
0 & \mbox{otherwise}
\end{cases}
\]
if $\chi_{v}$ is unramified and
\[
\Phi_{v}(x)=\begin{cases}
\chi_{v}^{-1}(x) & x\in r_{v}^{\times}\\
0 & \mbox{otherwise}
\end{cases}
\]
if $\chi_{v}$ is ramified. For a fractional ideal $\mathfrak{b}$
of $K$, we define $\Phi_{\chi,\mathfrak{b}}:\mathcal{S}(K(\mathbb{A}_{f}))$
by 
\[
\Phi_{\chi,\mathfrak{b}}(x)=\chi(b)\Phi_{\chi}(bx)
\]
where $b\in I(\mathfrak{b})$. It is obvious that $\Phi_{\chi,\mathfrak{b}}$
does not depend on the choice of $b$. Let $I_{\mathfrak{m}}$ be
the group of fractional ideals of $K$ relatively prime to $\mathfrak{m}$,
and let $\chi_{I}:I_{\mathfrak{m}}\to\mathbb{C}^{\times}$ be the
character of $I_{\mathfrak{m}}$ corresponding to $\chi$. Note that
$\chi_{I}$ is defined by 
\[
\chi_{I}(\mathfrak{a})=\prod_{\substack{v:{\rm finite}\\
v\nmid\mathfrak{m}
}
}\chi_{v}(a_{v})
\]
where $(a_{v})_{v}\in I(\mathfrak{a})$. For $x\in K^{\times}$ and
a fractional ideal $\mathfrak{b}$ of $K$, we have
\[
\chi_{\infty}^{-1}(x)\Phi_{\chi,\mathfrak{b}}(x)=\begin{cases}
\chi_{I}(x\mathfrak{b}) & x\in\mathfrak{b}^{-1}\mbox{ and }x\mathfrak{b}\in I_{m}\\
0 & \mbox{otherwise}.
\end{cases}
\]
Let $\rho=(\rho_{1},\dots,\rho_{n}):K(\mathbb{R})\to\mathbb{R}^{n}$
be a natural ring isomorphism, where $n$ is the degree of $K$. Let
$C$ be an ideal class of $K$. We define the partial Hecke $L$-functions
$L(s,\chi,C)$ by
\[
L(s,\chi,C)=\sum_{\substack{\mathfrak{a}\subset O_{K}\\
\mathfrak{a}\in C\cap I_{\mathfrak{m}}
}
}\chi_{I}(\mathfrak{a})N(\mathfrak{a})^{-s}.
\]
Assume that $\mathfrak{b}\in C$. Then we have
\begin{align}
L(s,\chi,C) & =\sum_{\mathfrak{a}}\chi_{I}(\mathfrak{a})N(\mathfrak{a})^{-s}\nonumber \\
 & =N(\mathfrak{b})^{-s}\sum_{\mathfrak{a}}\chi_{I}(\mathfrak{ab}^{-1}\mathfrak{b})N(\mathfrak{ab}^{-1})^{-s}\nonumber \\
 & =N(\mathfrak{b})^{-s}\sum_{x\in(\mathfrak{b}^{-1}\setminus\{0\})/O_{K}^{\times}}\chi(x\mathfrak{b})\left|\rho(x)\right|^{-s}\nonumber \\
 & =N(\mathfrak{b})^{-s}\sum_{x\in K^{\times}/O_{K}^{\times}}\chi_{\infty}^{-1}(x)\Phi_{\chi,\mathfrak{b}}(x)\left|\rho(x)\right|^{-s}.\label{eq:heckeL}
\end{align}

The Fourier transform of $\Phi_{\chi,\mathfrak{b}}$ is given as follows.
For $y_{0}\in I((\mathfrak{md})^{-1})$, we set 
\[
k(\chi)=\chi(y_{0})^{-1}\int_{K(\mathbb{A}_{f})}\psi(\mathrm{Tr}(xy_{0}))\Phi_{\chi}(x)dx.
\]
Note that $k(\chi)$ does not depend on the choice of $y_{0}$. Then
we have

\[
\hat{\Phi}_{\chi,\mathfrak{b}}=N(b)k(\chi)\Phi_{\chi^{-1},\mathfrak{mdb}^{-1}}.
\]

\subsection{\label{sub:regularization}Cassou-Nogu\`es' trick and the regularization}

Let $\Lambda\in C_{n}(V)$ be some cone. The problem is that $\Phi_{\chi,\mathfrak{b}}$
is not necessary regular with respect to $\Lambda$. In this section
we modify $\Phi_{\chi,\mathfrak{b}}$ to obtain a regular function.
Let $\mathbb{C}[I_{K}]$ be the group ring of the ideal group $I_{K}$
of $K$. We extend the definition of $\Phi_{\chi,\gamma}$ for $\gamma=\sum_{j}n_{j}\mathfrak{b}_{j}\in\mathbb{C}[I_{K}]$
where $n_{j}\in\mathbb{C}$ and $\mathfrak{b}\in I_{K}$ by
\[
\Phi_{\chi,\gamma}=\sum_{j}n_{j}\Phi_{\chi,\mathfrak{b}_{j}}.
\]

We use the notation $\hat{\gamma}=\sum_{j}n_{j}N(\mathfrak{b})\mathfrak{b}_{j}^{-1}\in\mathbb{C}[I_{K}]$
for $\gamma=\sum_{j}n_{j}\mathfrak{b}_{j}\in\mathbb{C}[I_{K}]$. Note
that the Fourier transform $\hat{\Phi}_{\chi,\gamma}$ of $\Phi_{\chi,\gamma}$
is given by
\[
\hat{\Phi}_{\chi,\gamma}=k(\chi)\Phi_{\chi^{-1},\mathfrak{md}\hat{\gamma}}.
\]
We say that a prime ideal $\mathfrak{p}$ is degree $m$ if $N(\mathfrak{p})=p^{m}$
where $p$ is a rational prime. For a fractional ideal $\mathfrak{a}=\prod_{v}\mathfrak{p}_{v}^{f_{v}}$
of $K$, we define a set $Q(\mathfrak{a})$ of rational primes by
\[
Q(\mathfrak{a})=\{p\mid\mbox{There exists a finite place }v\mbox{ such that }f_{v}\neq0\mbox{ and }\mathfrak{p}_{v}\mid p\}.
\]
The following is the same technique used in \cite{MR524276}.
\begin{lem}
Let $\mathfrak{b}$ a fractional ideal of $K$, and let $(u_{1},\dots,u_{n})\in O_{K}^{n}$
be a basis of $K$ as a vector space over $\mathbb{Q}$. We define
a finite set $Z$ of rational primes by
\[
Z=Q(\mathfrak{b})\cup Q(\mathfrak{m})\cup Q(u_{1})\cup\cdots\cup Q(u_{n}).
\]
Let $\mathfrak{p}$ be a prime ideal of degree $1$ such that $N(\mathfrak{p})\notin Z$.
Then the pair $(\Phi_{\chi,\gamma\mathfrak{b}},\Lambda(u_{1},\dots,u_{n}))$
satisfies Condition $P_{2}$, where
\[
\gamma=(1-N(\mathfrak{p})\chi_{I}(\mathfrak{p})\mathfrak{p}^{-1})\in\mathbb{C}[I_{K}].
\]
\end{lem}
\begin{proof}
Let $w\in K(\mathbb{A}_{f})$ and $u\in\{u_{1},\dots,u_{n}\}$. Put
$p=N(\mathfrak{p})$. We need to prove that
\[
\int_{\mathbb{A}_{f}}\Phi_{\chi,\gamma\mathfrak{b}}(w+xu)dx=0.
\]
There exists a positive integer $b$ such that $bu\in\mathfrak{mb}^{-1}$
and $b\notin\mathfrak{p}$. Note that 
\[
\Phi_{\chi,\mathfrak{bp}^{-1}}(x)=\begin{cases}
\chi_{I}(\mathfrak{p})^{-1}\Phi_{\chi,\mathfrak{b}}(x) & x\in\mathfrak{pb}^{-1}\\
0 & \mbox{otherwise}
\end{cases}
\]
for all $x\in K.$ Since $\Phi_{\chi,\gamma\mathfrak{b}}$ is invariant
under the translation by $pbu\hat{\mathbb{Z}}$, It is enough to prove
that
\[
\sum_{k=0}^{p-1}\Phi_{\chi,\gamma\mathfrak{b}}(w+kbu)=0.
\]
For all $k\in\mathbb{Z}$, we have
\[
\Phi_{\chi,\mathfrak{b}}(w+kbu)=\Phi_{\chi,\mathfrak{b}}(w).
\]
Since $bu\notin\mathfrak{pb}^{-1}$, there exists $l\in\{0,1,\dots,p-1\}$
such that for all $k\in\{0,\dots,p-1\}$ 
\[
p\chi_{I}(\mathfrak{p})\Phi_{\chi,\mathfrak{bp}}(w+kbu)=\begin{cases}
\Phi_{\chi,\mathfrak{b}}(w) & k=l\\
0 & \mbox{otherwise}.
\end{cases}
\]
Thus we have
\begin{align*}
\sum_{k=0}^{p-1}\Phi_{\chi,\gamma\mathfrak{b}}(w+kbu) & =\sum_{k=0}^{p-1}\Phi_{\chi,\mathfrak{b}}(w+kbu)-p\chi_{I}(\mathfrak{p})\sum_{k=0}^{p-1}\Phi_{\chi,\mathfrak{b}}(w+kbu)\\
 & =p\Phi_{\chi,\mathfrak{b}}(w)-p\Phi_{\chi,\mathfrak{b}}(w)\\
 & =0
\end{align*}
and the lemma is proved. \end{proof}
\begin{lem}
\label{lem:regularization-pre}Let $(u_{1},\dots,u_{n})\in O_{K}^{n}$
be a basis of $K$ over $\mathbb{Q}$. Let $(w_{1},\dots,w_{n})\in O_{K}^{n}$
be a basis of $K$ over $\mathbb{Q}$ such that $\mathrm{Tr}(u_{i}w_{j})=0$
for all $i\neq j$. We define a finite set $Z$ of rational primes
by
\[
Z=Q(\mathfrak{b})\cup Q(\mathfrak{m})\cup Q(\mathfrak{d})\cup\{p\in Q(u_{j})|1\leq j\leq n\}\cup\{p\in Q(w_{j})|1\leq j\leq n\}.
\]
Let $\mathfrak{p}$ and $\mathfrak{q}$ be prime ideals of degree
$1$ such that $N(\mathfrak{p})\notin Z$, $N(\mathfrak{q})\notin Z$
and $N(\mathfrak{p})\neq N(\mathfrak{q})$. Then the pair $(\Phi_{\chi,\gamma\mathfrak{b}},\Lambda(u_{1},\dots,u_{n}))$
satisfies Condition $P_{1}$ and $P_{2}$, where
\[
\gamma=(1-N(\mathfrak{p})\chi(\mathfrak{p})\mathfrak{p}^{-1})(1-\chi^{-1}(\mathfrak{q})\mathfrak{q})\in\mathbb{C}[I_{K}].
\]
\end{lem}
\begin{proof}
We set 
\[
\gamma_{1}=(1-N(\mathfrak{p})\chi(\mathfrak{p})\mathfrak{p}^{-1})\in\mathbb{C}[I_{K}],
\]
\[
\gamma_{2}=(1-N(\mathfrak{q})\chi^{-1}(\mathfrak{q})\mathfrak{q}^{-1})\in\mathbb{C}[I_{K}].
\]
Note that we have $\gamma=\gamma_{1}\hat{\gamma}_{2}$. By previous
lemma, $(\Phi_{\chi,\gamma\mathfrak{b}},\Lambda(u_{1},\dots,u_{n}))$
satisfies Condition $P_{2}$. Put $\mathfrak{a}=\gamma_{1}\mathfrak{b}$.
Note that the Fourier transform of $\Phi_{\chi,\gamma\mathfrak{b}}$
is given by $k(\chi)\Phi_{\chi^{-1},\gamma_{2}\mathfrak{md}\hat{\mathfrak{a}}}$.
By previous lemma, $(\Phi_{\chi^{-1},\gamma_{2}\mathfrak{md}\hat{\mathfrak{a}}},\Lambda(w_{1},\dots,w_{n}))$
satisfied Condition $P_{2}$. Therefore by Lemma \ref{Rem:p1p2},
$(\Phi_{\chi,\gamma\mathfrak{b}},\Lambda(u_{1},\dots,u_{n}))$ satisfy
Condition $P_{1}$. \end{proof}
\begin{defn}
We set 
\[
P(K)=\{(\mathfrak{p},\mathfrak{q})\mid\mathfrak{p},\mathfrak{q}\mbox{ are prime ideals of degree }1,\ N(\mathfrak{p})\neq N(\mathfrak{q})\}.
\]
We say that a statement $A(\mathfrak{p},\mathfrak{q})$ holds for
\emph{almost all} $(\mathfrak{p},\mathfrak{q})\in P(K)$ if there
exists a finite set $Z$ of rational primes such that $A(\mathfrak{p},\mathfrak{q})$
holds for
\[
\{(\mathfrak{p},\mathfrak{q})\mid N(\mathfrak{p})\notin Z,\ N(\mathfrak{q})\notin Z\}.
\]

\end{defn}
By Lemma \ref{lem:regularization-pre}, we have the following theorem.
\begin{thm}
\label{thm:regularization}Let $\chi$ be a Hecke character of $K$,
and let $z$ be an element of $\mathbb{C}[I_{K}]$. Let $T$ be a
finite set of simple cones. Then for almost all $(\mathfrak{p},\mathfrak{q})\in P(K)$,
\textup{$\Phi_{\chi,\gamma z}$ is regular with respect to $\Lambda$
for all $\Lambda\in T$, where} 
\[
\gamma=(1-N(\mathfrak{p})\chi(\mathfrak{p})\mathfrak{p}^{-1})(1-\chi^{-1}(\mathfrak{q})\mathfrak{q})\in\mathbb{C}[I_{K}].
\]

\end{thm}

\subsection{Fundamental domain and fundamental domain fan}

We use the following theorem by H. Yoshida, which can be found in
\cite{MR2011848} as Theorem 5.2 in chapter IV.
\begin{thm}
\label{Fact:exists-epsilon}There exists $\epsilon_{1},\dots,\epsilon_{n-1}\in E_{K}$
with the following properties.\end{thm}
\begin{enumerate}
\item $\epsilon_{1},\dots,\epsilon_{n-1}$ generate a subgroup $E$ of $E_{K}$
of finite index.
\item There exists a thin set $Z\subset\rho(K)\cap\mathbb{R}_{>0}^{n}$
and a fundamental domain $D\subset\mathbb{R}_{>0}^{n}$ such that
\[
\mathbb{R}_{>0}^{n}=\coprod_{\epsilon\in E}\epsilon D
\]
and
\[
\mathfrak{C}(F(1;\epsilon_{1},\dots,\epsilon_{n-1}))+\bm{1}_{Z}=\bm{1}_{D}.
\]

\end{enumerate}
We define the notion of a fundamental domain fan.
\begin{defn}
Let $\epsilon_{1},\dots,\epsilon_{n-1}$ and $E$ be as in Theorem
\ref{Fact:exists-epsilon}. Let $(\epsilon_{1}',\dots,\epsilon_{n-1}')$
be a generator of $E_{K}$. Then $\epsilon_{i}$ can be represented
as
\[
\epsilon_{i}=\prod_{j=1}^{n-1}(\epsilon_{j}')^{c_{ij}},
\]
where $C=(c_{ij})_{i,j}\in M_{n-1}(\mathbb{Z})$. We may assume that
$\det(C)=\#(E_{K}/E)$ by changing $\epsilon_{1}$ to $\epsilon_{1}^{-1}$
if necessary. We say that $\mathbb{D}$ is a \emph{fundamental domain
fan} if $\mathbb{D}$ is a fundamental domain fan for $(\epsilon_{1}',\dots,\epsilon_{n-1}')$.\end{defn}
\begin{rem}
For any basis $(\epsilon_{1}',\dots,\epsilon_{n-1}')$ of $E_{K}$,
$F(1;\epsilon_{1}',\dots,\epsilon_{n-1}')$ or $-F(1;\epsilon_{1}',\dots,\epsilon_{n-1}')$
is a fundamental domain fan. The above definition is to determine
this sign.
\end{rem}
For $g\in\{\pm1\}^{n}$, set $K_{g}=\{x\in K^{\times}\mid g_{\mu}\rho_{\mu}(x)>0\mbox{ for all }1\leq\mu\leq n\}$.
The following lemma follows from Theorem \ref{Fact:exists-epsilon}.
\begin{lem}
\label{lem:domain-sum}Let $g\in\{\pm1\}^{n}$ and $w\in K_{g}$.
Let $\epsilon_{1},\dots,\epsilon_{n-1}$ and $E$ be as in Theorem
\ref{Fact:exists-epsilon}. Then there exists a thin set $Z\subset K$
with the following property:

Let $f:K_{g}\to\mathbb{C}$ be any $E_{K}$-invariant function such
that 
\[
\sum_{x\in K_{g}/E_{K}}f(x)
\]
converges absolutely and $f(v)=0$ for all $v\in K_{g}\cap Z$. Then
we have
\[
\sum_{x\in K_{g}/E}f(x)=(\prod_{\mu}g_{\mu})\sum_{x\in K}f(x)\mathfrak{C}(F(w;\epsilon_{1},\dots,\epsilon_{n-1}))(x).
\]

\end{lem}

\subsection{Expression of L-function of general character of totally real field}

Let $(\rho_{\mu})_{\mu=1}^{n}$ be the tuple of all the embeddings
of $K$ to $\mathbb{R}$. For $\Phi\in\mathcal{S}(\mathbb{A}_{f}(K))$
and $\sigma:\{1,\dots,n\}\to\{0,1\}$, we define a function $\lambda:K^{\times}\to\mathbb{C}$
by
\[
\lambda(x)=\Phi(x)\prod_{\mu}\left(\frac{\rho_{\mu}(x)}{|\rho_{\mu}(x)|}\right)^{\sigma(\mu)}\ \ \ \ \ \ (x\in K^{\times})
\]
and a subset $X(\lambda)\subset\mathbb{C}^{n}$ by
\[
X(\lambda)=\{\bm{s}\in\mathbb{C}^{n}\mid\lambda(x\epsilon)=\lambda(x)\prod_{\mu=1}^{n}\left|\rho_{\mu}(\epsilon)\right|^{s_{\mu}}\mbox{ for any }(x,\epsilon)\in K^{\times}\times E_{K}\}.
\]
For $\bm{s}\in X(\lambda)$, we put
\[
H(\bm{s},\lambda)=\sum_{x\in K^{\times}/E_{K}}\frac{\lambda(x)}{\prod_{\mu=1}^{n}|\rho_{\mu}(x)|^{s_{\mu}}}.
\]

\begin{lem}
\label{lem:__wik_vanish}Let $\mathbb{B}\in W(K)+I_{n}(K,E_{K})+\ker\partial_{n}$.
There exists a finite set T with the following property:

Let $\Phi\in\mathcal{S}(K(\mathbb{A}_{f}))$ be any function which
is regular with respect to all $\Lambda\in T$. Then
\[
L_{g}(\bm{s},\Phi,\mathbb{B})=0
\]
 for all $g\in\{\pm1\}^{n}$ and $\bm{s}\in X(\lambda)$.\end{lem}
\begin{proof}
Let $\mathbb{B}=\mathbb{B}_{1}+\mathbb{B}_{2}$ where $\mathbb{B}_{1}\in W(K)+\ker\partial_{n}$
and $\mathbb{B}_{2}\in I_{n}(K,E_{K})$. We may assume that
\[
L_{g}(\bm{s},\Phi,\mathbb{B})=L_{g}(\bm{s},\Phi,\mathbb{B}_{1})+L_{g}(\bm{s},\Phi,\mathbb{B}_{2})
\]
by choosing a suitable $T$. By Lemma \ref{lem:sign_change_Lfunc}
and \ref{lem:W-Ker-shrink}, we may assume that
\[
L_{g}(\bm{s},\Phi,\mathbb{B}_{1})=L_{g}(\bm{s},\Phi,P(\mathbb{B}_{1}))=0.
\]
Let $\mathbb{B}_{2}=\sum_{j\in J}(\epsilon_{j}\mathbb{B}_{j}-\mathbb{B}_{j})$
where $\epsilon_{j}\in E_{K}$ and $\mathbb{B}_{j}\in C_{n}(K)$.
Since $\bm{s}\in X(\lambda)$, we may assume that
\[
L_{g}(\bm{s},\Phi,\epsilon_{j}\mathbb{B}_{j})=L_{g}(\bm{s},\Phi,\mathbb{B}_{j})
\]
for all $j\in J$. Thus the lemma is proved.\end{proof}
\begin{prop}
\label{prop:const_of_L_of_general_char}Let $\mathbb{D}$ be a fundamental
domain fan. Then there exists a finite set $T$ with the following
property:

Let $\Phi\in\mathcal{S}(\mathbb{A}_{f}(K))$ be any function which
is regular with respect to all $\Lambda\in T$. Then 
\[
H(\bm{s},\lambda)=L_{\sigma}(\bm{s},\Phi,\mathbb{D})
\]
 for all $\sigma$ and $\bm{s}\in X(\lambda)$.\end{prop}
\begin{proof}
We have
\begin{align*}
H(\bm{s},\lambda) & =\sum_{g\in\{\pm1\}^{n}}\sum_{x\in K_{g}/E_{K}}\frac{\lambda(x)}{\prod_{\mu=1}^{n}|\rho_{\mu}(x)^{s_{\mu}}|}\\
 & =\sum_{g\in\{\pm1\}^{n}}(\prod_{\mu}g_{\mu}^{\sigma(\mu)})\sum_{x\in K_{g}/E_{K}}\frac{\Phi(x)}{\prod_{\mu=1}^{n}(g_{\mu}\rho_{\mu}(x))^{s_{\mu}}}.
\end{align*}
For $g\in\{\pm1\}$, let $w(g)\in K_{g}$ be any element. Then by
Lemma \ref{lem:domain-sum}, we may assume that
\begin{align*}
H(\bm{s},\lambda) & =\sum_{g\in\{\pm1\}^{n}}(\prod_{\mu}g_{\mu}^{\sigma(\mu)})\sum_{x\in K_{g}/E_{K}}\frac{\Phi(x)}{\prod_{\mu=1}^{n}(g_{\mu}\rho_{\mu}(x))^{s_{\mu}}}\\
 & =\frac{1}{\#(E_{K}/E)}\sum_{g\in\{\pm1\}^{n}}(\prod_{\mu}g_{\mu}^{1-\sigma(\mu)})\sum_{x\in K_{g}}\frac{\Phi\mathfrak{C}(F(w(g);\epsilon_{1},\dots,\epsilon_{n-1}))(x)}{\prod_{\mu=1}^{n}(g_{\mu}\rho_{\mu}(x))^{s_{\mu}}}
\end{align*}
by choosing a suitable $T$. By Lemma \ref{lem:series-expression},
we may assume that
\begin{align*}
H(\bm{s},\lambda) & =\frac{1}{\#(E_{K}/E)}\sum_{g\in\{\pm1\}^{n}}(\prod_{\mu}g_{\mu}^{1-\sigma(\mu)})\sum_{x\in K_{g}}\frac{\Phi\mathfrak{C}(F(w(g);\epsilon_{1},\dots,\epsilon_{n-1}))(x)}{\prod_{\mu=1}^{n}(g_{\mu}\rho_{\mu}(x))^{s_{\mu}}}\\
 & =\frac{1}{\#(E_{K}/E)}\sum_{g\in\{\pm1\}^{n}}(\prod_{\mu}g_{\mu}^{\sigma(\mu)})L_{g}(\bm{s},\Phi,F(w(g);\epsilon_{1},\dots,\epsilon_{n-1})).
\end{align*}
Note that we have 
\[
F(w(g);\epsilon_{1},\dots,\epsilon_{n-1})-\#(E_{K}/E)\mathbb{D}\in W(K)+I(K,E_{K})+\ker\partial_{n}
\]
by Lemma \ref{lem:another_basis}. By Lemma \ref{lem:__wik_vanish},
we may assume that
\begin{align*}
H(\bm{s},\lambda) & =\frac{1}{\#(E_{K}/E)}\sum_{g\in\{\pm1\}^{n}}(\prod_{\mu}g_{\mu}^{\sigma(\mu)})L_{g}(\bm{s},\Phi,F(w(g);\epsilon_{1},\dots,\epsilon_{n-1}))\\
 & =\sum_{g\in\{\pm1\}^{n}}(\prod_{\mu}g_{\mu}^{\sigma(\mu)})L_{g}(\bm{s},\Phi,\mathbb{D})\\
 & =L_{\sigma}(\bm{s},\Phi,\mathbb{D})
\end{align*}
\end{proof}
\begin{rem}
It is natural to expect that if $\Phi$ is regular for $\mathbb{D}$
then $H(\bm{s},\lambda)=L_{\sigma}(\bm{s},\Phi,\mathbb{D})$, but
we do not know a proof.
\end{rem}

\subsection{Expression of Hecke L-function}

Let $\chi$ be a Hecke character of $K$. We define $\sigma:\{1,\dots,n\}\rightarrow\{0,1\}$
and $h_{\mu}\in i\mathbb{R}$ by
\[
\chi_{\infty}^{-1}(t)=\prod_{\mu}\left(\frac{t_{\mu}}{|t_{\mu}|}\right)^{\sigma(\mu)}|t_{\mu}|^{-h_{\mu}}.
\]
For $z=\sum_{j}n_{j}\mathfrak{b}_{j}\in\mathbb{C}[I_{K}]$, we define
$L(s,\chi,z)$ by
\[
L(s,\chi,z)=\sum_{j}n_{j}N(\mathfrak{b}_{j})^{s}L(s,\chi,C_{j})
\]
where $C_{j}$ is an ideal class such that $\mathfrak{b}_{j}\in C_{j}$,
and $c(z,s)$ by
\[
c(z,s)=\sum_{j}n_{j}N(\mathfrak{b}_{j})^{s}.
\]
Let $C$ be an ideal class. Then by (\ref{eq:heckeL}), for all $\mathfrak{b}\in C$,
we have
\begin{align*}
\#(O_{K}^{\times}/E_{K})L(s,\chi,C) & =N(\mathfrak{b})^{-s}\sum_{x\in K^{\times}/E_{K}}\chi_{\infty}^{-1}(x)\phi_{\chi,\mathfrak{b}}(x)\left|\rho(x)\right|^{-s}\\
 & =N(\mathfrak{b})^{-s}H((s+h_{1},\dots,s+h_{n}),\lambda)
\end{align*}
where $\lambda$ is defined by 
\[
\lambda=\Phi_{\chi,\mathfrak{b}}(x)\prod_{\mu}\left(\frac{\rho_{\mu}(x)}{|\rho_{\mu}(x)|}\right)^{\sigma(\mu)}.
\]
Let $\mathbb{D}\in C_{n}(K)$ be a fundamental domain fan. Then by
Proposition \ref{prop:const_of_L_of_general_char}, for almost all
$(\mathfrak{p},\mathfrak{q})\in P(K)$ we have 
\begin{align*}
\#(O_{K}^{\times}/E_{K})L(s,\chi,\gamma\mathfrak{b}) & =H((s+h_{1},\dots,s+h_{n}),\Phi_{\chi,\gamma\mathfrak{b}},\sigma)\\
 & =L_{\sigma}((s+h_{1},\dots,s+h_{n}),\Phi_{\chi,\gamma\mathfrak{b}},\sigma)
\end{align*}

\begin{defn}
Let $\chi$ be a Hecke character of $K$, and let $\mathbb{D}$ be
a fundamental domain fan. Let $\gamma$$\in\mathbb{C}[I_{K}]$. Assume
that $\Phi_{\chi,\gamma}\in\mathcal{R}(K,\mathbb{D})$. We define
a function $F_{\chi,\gamma,\mathbb{D}}:\mathbb{C}^{n}\to\mathbb{C}$
of several variables by
\[
F_{\chi,\gamma,\mathbb{D}}(s_{1},\dots,s_{n})=\frac{1}{\#(O_{K}^{\times}/E_{K})}L_{\sigma}((s_{\mu}+h_{\mu})_{\mu},\Phi_{\chi,\gamma},\mathbb{D})
\]
where $(h_{\mu})_{\mu}\in(i\mathbb{R})^{n}$ and $\sigma:\{1,\dots,n\}\to\{0,1\}$
are defined by 
\[
\chi_{\infty}(t)=\prod_{\mu=1}^{n}\left(\frac{t_{\mu}}{|t_{\mu}|}\right)^{\sigma(\mu)}|t_{\mu}|^{-h_{\mu}}.
\]

\end{defn}
Then we have the following theorem.
\begin{thm}
\label{thm:main_thm}Let $K$ be a totally real field of degree $n$,
and let $\chi$ be a Hecke character of $K$. Let $\mathbb{D}$ be
a fundamental domain fan, $C$ an ideal class, and $\mathfrak{b}\in C$
an ideal. For $\mathfrak{p},\mathfrak{q}\in I_{k}$, we set 
\[
\gamma=(1-N(\mathfrak{p})\chi_{I}(\mathfrak{p})\mathfrak{p}^{-1})(1-\chi_{I}^{-1}(\mathfrak{q})\mathfrak{q})\in\mathbb{C}[I_{K}]
\]
 and $f=\gamma\mathfrak{b}$. Then for almost all $(\mathfrak{p},\mathfrak{q})\in P(K)$,
the following assertions hold:
\begin{enumerate}
\item $\Phi_{\chi,f}$ is regular with respect to $\mathbb{D}$.
\item The function $F_{\chi,f,\mathbb{D}}(s_{1},\dots,s_{n})$ is holomorphic
for all $\bm{s}\in\mathbb{C}^{n}$.
\item The function $F_{\chi,f,\mathbb{D}}(s_{1},\dots,s_{n})$ has the following
functional equation
\begin{multline*}
\Gamma_{\sigma}((s_{\mu}+h_{\mu})_{\mu})F_{\chi,f,\mathbb{D}}(s_{1},\dots,s_{r})\\
=i_{\chi}k(\chi)\Gamma_{\sigma}((1-s_{\mu}-h_{\mu}))F_{\chi^{-1},\mathfrak{md}\widehat{f},\varphi(\mathbb{D})}(1-s_{1},\dots,1-s_{r}).
\end{multline*}

\item The function $F_{\chi,f,\mathbb{D}}(s_{1},\dots,s_{n})$ has zeros
at $s_{\mu}-(k+h_{\mu})=0$ for all $1\leq\mu\leq r$ and non-positive
integer $k$ such that $k\equiv\sigma(\mu)$ $\pmod{2}$.
\item The diagonal parts of $F_{\chi,f,\mathbb{D}}$ and $F_{\chi^{-1},\mathfrak{md}\hat{f},\varphi(\mathbb{D})}$
are Hecke L-functions, i.e., we have 
\[
F_{\chi,f,\mathbb{D}}(s,\dots,s)=L(s,\chi,f)
\]
and 
\[
F_{\chi^{-1},\mathfrak{md}\hat{f},\varphi(\mathbb{D})}(1-s,\dots,1-s)=L(1-s,\chi^{-1},\mathfrak{md}\hat{f}).
\]

\item If $\mathfrak{p}$ and $\mathfrak{q}$ are principal ideals then we
have
\[
F_{\chi,f,\mathbb{D}}(s,\dots,s)=c(f,s)L(s,\chi,C)
\]
and
\[
F_{\chi^{-1},\mathfrak{md}\hat{f},\varphi(\mathbb{D})}(1-s,\dots,1-s)=c(\mathfrak{md}\hat{f},s)L(1-s,\chi^{-1},\mathfrak{md}C^{-1}).
\]

\end{enumerate}
\end{thm}
\begin{rem}
There exist infinitely many principal prime ideals of degree $1$
of $K$ by Chebotarev's density theorem.
\end{rem}

\subsection{\label{sub:Functional-equation-of-Hecke}Functional equation of Hecke
L-function}

In this section, we give a new proof of the functional equation of
the Hecke L-functions by Theorem \ref{thm:main_thm}. Let us state
the functional equation of partial Hecke L-functions. We set $\mathfrak{m}$
be the conductor of $\chi$ and $\mathfrak{d}$ be the different of
$K$. We set 
\[
W_{\chi}=i_{\sigma}k(\chi)\sqrt{N(\mathfrak{md})}.
\]
For an ideal class $C$ of $K$, we define the complete partial Hecke
L-function by
\[
\hat{L}(s,\chi,C)=N(\mathfrak{md})^{s/2}\Gamma_{\chi}(s)L(s,\chi,C)
\]
where $\Gamma_{\chi}(s)=\Gamma_{\sigma}((s+h_{\mu})_{\mu})$. Then
the functional equation of the partial Hecke L-function is as follows
\begin{cor}
Let $C$ and $C'$ be ideal classes such that $CC'=(\mathfrak{md})$
where $(\mathfrak{md})$ is an ideal class which contains $\mathfrak{md}.$
We have
\[
\hat{L}(s,\chi,C)=W_{\chi}\hat{L}(1-s,\chi^{-1},C').
\]
\end{cor}
\begin{proof}
Take any ideal $\mathfrak{b}$ of $K$ in $C$. Let $\mathfrak{p}$
and $\mathfrak{q}$ be principal prime ideals of $K$ in Theorem \ref{thm:main_thm}.
Then we have
\[
c(f,s)\Gamma_{\chi}(s)L(s,\chi,C)=i_{\chi}k_{\chi}c(\mathfrak{md}\widehat{f},1-s)\Gamma_{\chi^{-1}}(1-s)L(1-s,\chi^{-1},C')
\]
where $f=(1-N(\mathfrak{p})\chi_{I}(\mathfrak{p})\mathfrak{p}^{-1})(1-\chi_{I}^{-1}(\mathfrak{q})\mathfrak{q})\mathfrak{b}\in\mathbb{C}[I_{K}]$.
Since $c(f,s)=c(\hat{f},1-s)\neq0$, we get the corollary by
\begin{align*}
\Gamma_{\chi}(s)L(s,\chi,C) & =i_{k}k(\chi)N(\mathfrak{md})^{1-s}\Gamma_{\chi^{-1}}(1-s)L(1-s,\chi^{-1},C'),\\
N(\mathfrak{md})^{s/2}\Gamma_{\chi}(s)L(s,\chi,C) & =i_{k}k(\chi)N(\mathfrak{md})^{1/2}N(\mathfrak{md})^{(1-s)/2}\Gamma_{\chi^{-1}}(1-s)L(1-s,\chi^{-1},C'),\\
\hat{L}(s,\chi,C) & =W_{\chi}\hat{L}(1-s,\chi^{-1},C').
\end{align*}

\end{proof}
\bibliographystyle{plain}
\bibliography{reference}

\end{document}